\documentclass{amsart}

\usepackage{amsmath,amssymb}
\usepackage{amsthm}
\usepackage{amsrefs}
\usepackage{indentfirst}
\usepackage{mathrsfs}
\usepackage{hyperref}
\usepackage{xcolor}
\usepackage[margin=1.4in]{geometry}

\usepackage{mathtools}
\usepackage{algorithm}
\usepackage{algorithmicx}
\usepackage{algpseudocode}
\usepackage{enumitem}

\newlength{\leftstackrelawd}
\newlength{\leftstackrelbwd}
\def\leftstackrel#1#2{\settowidth{\leftstackrelawd}%
	{${{}^{#1}}$}\settowidth{\leftstackrelbwd}{$#2$}%
	\addtolength{\leftstackrelawd}{-\leftstackrelbwd}%
	\leavevmode\ifthenelse{\lengthtest{\leftstackrelawd>0pt}}%
	{\kern-.5\leftstackrelawd}{}\mathrel{\mathop{#2}\limits^{#1}}}

\numberwithin{equation}{section}

\theoremstyle{plain}
\newtheorem{maintheorem}{Theorem}
\newtheorem{theorem}{Theorem}[section]
\newtheorem{lemma}[theorem]{Lemma}
\newtheorem{proposition}[theorem]{Proposition}
\newtheorem{corollary}[theorem]{Corollary}
\newtheorem{definition}[theorem]{Definition}
\newtheorem{assumption}{Assumption}
\newtheorem{remark}[theorem]{Remark}

\newcommand{\bR}{\mathbb{R}}
\newcommand{\bNp}{\mathbb{N}_+}
\newcommand{\bN}{\mathbb{N}}
\newcommand{\bZ}{\mathbb{Z}}
\newcommand{\bE}{\mathbb{E}}
\newcommand{\bP}{\mathbb{P}}

\newcommand{\calC}{\Omega_s}
\newcommand{\calA}{\mathrm{Crit}_s(f)}
\newcommand{\crit}{\mathrm{Crit}(f)}

\newcommand{\norm}[1]{\left\Vert #1\right\Vert}
\DeclareMathOperator{\ran}{ran}

\title[Global convergence of randomized coordinate gradient descent]{On the global convergence of randomized coordinate gradient descent for non-convex optimization}
\author{Ziang Chen}
\address{(ZC) Department of Mathematics, Duke University, Box 90320, Durham, NC 27708, USA.}
\email{ziang@math.duke.edu}
\author{Yingzhou Li}
\address{(YL) School of Mathematical Sciences, Fudan University, Shanghai 200433, China}
\email{yingzhouli@fudan.edu.cn}
\author{Jianfeng Lu}
\address{(JL) Departments of Mathematics, Physics, and Chemistry, Duke University, Box 90320, Durham, NC 27708, USA.}
\email{jianfeng@math.duke.edu}
\date{\today}

\thanks{This work is supported in part by the National Science Foundation via grants DMS-2012286 and CHE-2037263, and by US Department of Energy via grant DE-SC0019449. Y. Li is partially supported by National Natural Science Foundation of China under Grant No. 12271109. We thank Jonathan Mattingly, Zhe Wang, and Stephen J.{} Wright for helpful discussions.}
\begin{document}

	\begin{abstract}
		In this work, we analyze the global convergence property of coordinate gradient descent with random choice of  coordinates and stepsizes for non-convex optimization problems. Under generic assumptions, we prove that the algorithm iterate will almost surely escape strict saddle points of the objective function. As a result, the algorithm is guaranteed to converge to local minima if all saddle points are strict. Our proof is based on viewing coordinate descent algorithm as a nonlinear random dynamical system and a quantitative finite block analysis of its linearization around saddle points.
	\end{abstract}
	
	\maketitle
	
	\section{Introduction}
	In this paper, we analyze the global convergence of coordinate gradient descent algorithm for smooth but non-convex optimization problem
	\begin{equation}\label{obj-fun}
		\min_{x\in \bR^d} f(x).
	\end{equation}
	More specifically, we consider coordinate gradient descent with random coordinate selection and random stepsizes, as Algorithm~\ref{alg:random CGD}.
	\begin{algorithm}[H]
		\caption{Randomized coordinate gradient descent}
		\label{alg:random CGD}
		\begin{algorithmic}
			\State Initialization: $x_0\in\bR^d$, $t=0$.
			\While {not convergent}
			\State Draw a coordinate $i_t$ uniformly random from $\{1,2,\dots,d\}$.
			\State Draw a stepsize $\alpha_t$ uniformly random in $[\alpha_{\min},\alpha_{\max}]$.
			\State $x_{t+1}\leftarrow x_t- \alpha_t e_{i_t} \partial_{i_t} f(x_t)$.
			\State $t\leftarrow t+1$.
			\EndWhile
		\end{algorithmic}
	\end{algorithm}
	The main result of this paper, Theorem~\ref{thm: main},  is that for any initial guess $x_0$ that is not a strict saddle point of $f$, under some mild conditions, with probability $1$, Algorithm~\ref{alg:random CGD} will escape any strict saddle points; and thus under some additional structural assumption of $f$, the algorithm will globally converge to a local minimum. 
	
	In order to establish the global convergence, we view the algorithm as a random dynamical system and carry out the  analysis based on the theory of random dynamical systems. This might be of separate interest, in particular, to the best of our knowledge, the theory of random dynamical system has not been utilized in analyzing randomized algorithms; while it offers a natural framework to establish long time behavior of such algorithms. Let us now briefly explain the random dynamical system view of the algorithm and our analysis; more details can be found in Section~\ref{sec: main result}. 
	
	\smallskip

	Let $(\Omega,\mathcal{F},\bP)$ be the probability space for all randomness used in the algorithm, such that each $\omega\in\Omega$ is a sequence of coordinates and stepsizes. The iterate of Algorithm~\ref{alg:random CGD} can be described as a random  dynamical system  $x_t=\varphi(t,\omega)x_0$ where $\varphi(t,\omega):\bR^d\rightarrow\bR^d$ is a nonlinear map for any given $t\in\bN$ and $\omega\in\Omega$. 
	
	Consider an isolated stationary point $x^*$ of the dynamical system, which corresponds to a critical point of $f$. Near $x^*$, the dynamical system can be approximated by its linearization: $x_t=\Phi(t,\omega)x_0$, where $\Phi(t,\omega)\in\bR^{d\times d}$. The limiting behavior of the linear dynamical system can be well understood by the celebrated multiplicative ergodic theorem: Under some assumptions, the limit $\Lambda(\omega)=\lim_{t\rightarrow \infty}\left(\Phi(t,\omega)^{\top} \Phi(t,\omega)\right)^{1/2t}$ exists  almost surely. The eigenvalues of the matrix $\Lambda(\omega)$,  $e^{\lambda_1(\omega)}>e^{\lambda_2(\omega)}>\cdots>e^{\lambda_{p(\omega)}(\omega)}$, characterize the long time behavior of the system. 
	In particular, if the largest Lyapunov exponent $\lambda_1(\omega)$ is strictly positive, then if $x_0$ has some non-trivial component in the unstable subspace,  $x_t=\Phi(t,\omega)x_0$ would exponentially diverge from $x^*$. More details of preliminaries of linear random dynamical system can be found in Section~\ref{sec: RDS}.
	
	Intuitively, one expects that the nonlinear dynamical system can be approximated by its linearization around a critical point $x^*$, and would hence escape the strict saddle point, following the linearized system. However, the approximation by linear dynamical system cannot hold for infinite time horizon, due to error accumulation. Therefore, we cannot naively conclude using the multiplicative ergodic theorem and the linear approximation. Instead, a major part of analysis is devoted to establish a quantitative finite block analysis of the behavior of the dynamical system over finite time interval. In particular, we will prove that when the iterate is in a neighborhood of $x^*$, the distance $\norm{x_t - x^*}$ will
	be exponentially amplified for a duration $T$ with high probability. This would then be used to prove that with probability $1$ the nonlinear system will escape strict saddle points. 
	
	\subsection{Related work} 
	Coordinate gradient descent is a popular approach in optimization, see e.g., review articles \cite{Wright-15, shi2016primer}. Advantages of coordinate gradient method include that compared with the full gradient descent, it allows larger stepsize \cite{Nesterov-12} and enjoys faster convergence \cite{Saha-13}, and it is also friendly for parallelization \cite{Liu-15, Richtarik-16}. 
	
	The convergence of coordinate gradient descent has been analyzed in several settings on the property of the objective function and on the strategy of coordinate selection. 
	The understanding of convergence for convex problems is quite complete: For methods with cyclic choice of coordinates, the convergence has been established in \cite{Beck-13, Saha-13, Sun-15} and the worst-case complexity is investigated when the objective function is convex and quadratic in \cite{Sun-19}. For methods with random choice of coordinates, it is shown in \cite{Nesterov-12} that $\bE f(x_t)$ converges to $f^*=\min_{x\in\bR^d}f(x)$ sublinearly in the convex case and linearly in the strongly convex case. Convergence of objective function in high probability has also been established in  \cite{Nesterov-12}. We also refer to \cite{Richtarik-14, Liu-14, Liu-15, Wright-15} for further convergence results for random coordinate selection for convex problems. 
	More recently, convergence of methods with random permutation of coordinates (i.e., a random permutation of the $d$ coordinates is used for every $d$ step of the algorithm) have been analyzed, mostly for the case of quadratic objective functions \cite{Lee-19, Oswald-17, Gurbuzbalaban-20, Wright-20}. It has been an ongoing research direction to compare various coordinate selection strategies in various settings. In addition, in the non-convex and non-smooth setting, the convergence of coordinate/alternating descent methods can be analyzed for tame/semi-algebraic functions with Kurdyka-{\L}ojasiewicz property (see e.g. \cite{attouch2013convergence, attouch2010proximal, bolte2014proximal, boct2020inertial}).
	
	For non-convex objective functions, the global convergence analysis is less developed, as the situation becomes more complicated. Escaping strict saddle points has been a focused research topic in non-convex optimization, motivated by applications in machine learning. It has been established that various first-order algorithms with gradient noise or added randomness to iterates would escape strict saddle points, see e.g., \cite{Ge-15, Levy-16, Jin-17, Jin-18, Jin-19, Guo-20} for works in this direction.
	
	Among previous works for escaping saddle points, perhaps the closest in spirit to our current result are \cite{Lee-16, ONeill-19, Lee-2019, leadeigen-19}, where algorithms without gradient or iterate randomness are studied.
	It is proved in \cite{Lee-16} that for almost every initial guess, the trajectory of the gradient descent algorithm (without any randomness) with constant stepsize would not converge to a strict saddle point. The result has been extended in \cite{Lee-2019} to a broader class of deterministic first-order algorithms, including coordinate gradient descent with cyclic choice of coordinate. The global convergence result for cyclic coordinate gradient descent is also proved in \cite{leadeigen-19} under slightly more relaxed conditions. Similar convergence result is also obtained for heavy-ball method in \cite{ONeill-19}. Let us emphasize that in the case of coordinate algorithms, it is not merely a technical question whether the algorithm can escape the strict saddle points without randomly perturbing gradients or iterates. In fact, one simply cannot employ such random perturbations, e.g., adding a random Gaussian vector to the iterate, since doing so would destroy the coordinate nature of the algorithm. 
	
	The analysis in works 
	\cite{Lee-16,Lee-2019,ONeill-19,leadeigen-19} is based on viewing the algorithm as a deterministic dynamical system, and applying center-stable manifold theorem for deterministic dynamical system \cite{Shub-87}, which characterizes the local behavior near a stationary point of nonlinear dynamical systems. Such a framework obviously does not work for randomized algorithms. 
	To some extent, our analysis can be understood as a natural generalization to the framework of random dynamical systems, which allows us to analyze the long time behavior of randomized algorithms, in particular coordinate gradient descent with random coordinate selection.  
	
	Let us mention that various stable, unstable, and center manifolds theorems have been established in the literature of random dynamical systems, see e.g., \cite{Arnold_RDS, Ruelle-1979, Ruelle-82, Boxler-89-center, Liu-06}. These sample-dependent random manifolds also characterize the local behavior of random dynamical systems. However, as far as we can tell, one cannot simply apply such ``off-the-shelf results'' for the analysis of Algorithm~\ref{alg:random CGD}. Instead, for study of the algorithm, we have to carry out a quantitative finite block analysis for the random dynamical system near the stationary points. Our proof technique is inspired by stability analysis of Lyapunov exponent of random dynamical systems, as in  \cite{Ledrappier-91, Froyland-15}.

	\subsection{Organization} The rest of this paper will be organized as follows. In Section~\ref{sec: RDS}, we review the preliminaries of random dynamical system, for convenience of readers. Our main result is stated in Section~\ref{sec: main result}. The proofs can be found in Section~\ref{sec: proof}.

	\section{Preliminaries of random dynamical systems}\label{sec: RDS}
	
	In this section, we recall basic notions and results of random dynamical systems, for more details, we refer the readers to standard reference, such as \cite{Arnold_RDS}. After introducing the preliminaries in this section, we will define the random dynamical system associated with Algorithm~\ref{alg:random CGD} in Section~\ref{sec: notation}.
	Let $(\Omega,\mathcal{F},\bP)$ be a probability space and let $\mathbb{T}$ be a semigroup with $\mathcal{B}(\mathbb{T})$ being its Borel $\sigma$-algebra. $\mathbb{T}$ serves as the notion of time. In the setting of Algorithm~\ref{alg:random CGD}, we have $\mathbb{T}=\bN$, corresponding to the one-sided discrete time setting. Other possible examples of $\mathbb{T}$ include $\mathbb{T}=\bZ$, $\mathbb{T}=\bR_{\geq 0}$, and $\mathbb{T}=\bR$, with the assumption that $0\in\mathbb{T}$.
	
	Let us first define random dynamical system. As we have mentioned in the introduction, the dynamics starting from $x_0$ can be determined once a sample $\omega\in\Omega$ is fixed. From the viewpoint of random dynamical system, specifying the dynamics of $x$ is equivalent to specifying the dynamics of $\omega$: Suppose at time $0$, the dynamics corresponds to $\omega$, then to prescribe the future dynamics starting from time $t$, we can specify the corresponding $\theta(t) \omega \in \Omega$ for some map $\theta(t): \Omega \to \Omega$. More precisely, we have the following definition of dynamics on $\Omega$. 
	\begin{definition}[Metric dynamical system] A metric dynamical system on a probability space $(\Omega,\mathcal{F},\bP)$ is a family of maps $\{\theta(t):\Omega\rightarrow\Omega\}_{t\in\mathbb{T}}$ satisfying that
		\begin{itemize}
			\item[(i)] The mapping $\mathbb{T}\times\Omega\rightarrow\Omega,\ (t,\omega)\mapsto\theta(t)\omega$ is measurable;
			
			\item[(ii)] It holds that $\theta(0)=\mathrm{Id}_{\Omega}$ and $\theta(t+s)=\theta(t)\circ\theta(s),\ \forall\ s,t\in\mathbb{T}$;
			
			\item[(iii)] $\theta(t)$ is $\bP$-preserving for any $t\in\mathbb{T}$, where we say a map $\theta:\Omega\rightarrow\Omega$ is $\bP$-preserving if 
			\begin{equation*}
				\bP(\theta^{-1}B)=\bP(B),\quad \forall\ B\in\mathcal{F}.
			\end{equation*}
		\end{itemize}
	\end{definition}
	
	The random dynamical system can then be defined as follows. 
	
	\begin{definition}[Random dynamical system]
		Let $(X,\mathcal{F}_X)$ be a measurable space and let $\{\theta(t):\Omega\rightarrow\Omega\}_{t\in\mathbb{T}}$ be a metric dynamical system on $(\Omega,\mathcal{F},\bP)$. Then a random dynamical system on $(X,\mathcal{F}_X)$ 
		over $\{\theta(t)\}_{t\in\mathbb{T}}$ is a measurable map 
		\begin{equation*}
			\begin{split}
				\varphi:\mathbb{T}\times\Omega\times X&\rightarrow\quad X,\\
				(t,\omega,x)\ \  & \mapsto \varphi(t,\omega,x),
			\end{split}
		\end{equation*}
		satisfying the following cocycle property: for any $\omega\in\Omega$, $x\in X$, and $s,t\in \mathbb{T}$, it holds that 
		\begin{equation*}
			\varphi(0,\omega,x)=x,
		\end{equation*}
		and that
		\begin{equation}\label{cocycle}
			\varphi(t+s,\omega,x)=\varphi(t,\theta(s)\omega,\varphi(s,\omega,x)).
		\end{equation}
	\end{definition}
	
	The cocycle property \eqref{cocycle} is a key property of random dynamical system: After time $s$, if we restart the system at $x_s$, the future dynamic corresponds to the sample $\theta(s)\omega$. Note that $\varphi(t,\omega,\cdot)$ is a map on $X$, with some ambiguity of notation, we also use $\varphi(t, \omega)$ to denote this map on $X$ and write $\varphi(t,\omega)x=\varphi(t,\omega,x)$. Then the cocycle property \eqref{cocycle} can be written as 
	\begin{equation*}
		\varphi(t+s,\omega)=\varphi(t,\theta(s)\omega)\circ\varphi(s,\omega).
	\end{equation*}
	
	In this work, we will focus on the one-sided discrete time $\mathbb{T}=\bN$ and $\theta(t)=\theta^t$, where $\theta$ is $\bP$-preserving and $\theta^t$ is the $t$-fold composition of $\theta$. 
	Suppose that $X=\bR^d$ and $A:\Omega\rightarrow\text{GL}(d,\bR)$ is measurable. Consider a linear random dynamical system defined as (we use $\Phi$ for the linear system, while reserving $\varphi$ for nonlinear dynamics considered later)
	\begin{equation*}
		\Phi(t,\omega)=A(\theta^{t-1}\omega)\cdots A(\theta\omega)A(\omega),
	\end{equation*}
	where the right-hand side is the product of a sequences of random matrices. In this setting, the behavior of the linear system $x_t=\Phi(t,\omega)x_0$ is well understood by the celebrated multiplicative ergodic theorem, also known as the Oseledets theorem, which we recall in Theorem~\ref{thm: MET}. Such type of results was first established by V.I.{} Oseledets~\cite{oseledets} and was further developed in many works such as \cite{raghunathan, Ruelle-1979, Walters-93}.
	
	\begin{theorem}[Multiplicative ergodic theorem, {\cite[Theorem 3.4.1]{Arnold_RDS}}]
		\label{thm: MET}Suppose that 
		\begin{equation*}
			\left(\log\norm{A(\cdot)}\right)_+,\ \left(\log\norm{A(\cdot)^{-1}}\right)_+\in L^1(\Omega,\mathcal{F},\bP),
		\end{equation*}
		where we have used the short-hand $a_+:=\max\{a,0\}$. Then there exists an $\theta$-invariant $\widetilde{\Omega}\in \mathcal{F}$ with $\bP(\widetilde{\Omega})=1$, such that the followings hold for any $\omega\in\widetilde{\Omega}$:
		\begin{itemize}
			\item[(i)] It holds that the limit
			\begin{equation}\label{conv: Lamb}
				\Lambda(\omega)=\lim_{t\rightarrow\infty}\left(\Phi(t,\omega)^{\top}\Phi(t,\omega)\right)^{1/2t},
			\end{equation}
			exists and is a positive definite matrix. Here $\Phi(t, \omega)^{\top}$ denotes the transposition of the matrix (as $\Phi(t, \omega)$ is a linear map on $X$).
			\item[(ii)] Suppose  $\Lambda(\omega)$ has $p(\omega)$ distinct eigenvalues, which are ordered as $e^{\lambda_1(\omega)}>e^{\lambda_2(\omega)}>\cdots> e^{\lambda_{p(\omega)}(\omega)}$. Denote $V_i(\omega)$ the corresponding eigenspace, with dimension $d_i(\omega)$, for  $i=1,2,\dots,p(\omega)$. Then the functions $p(\cdot)$, $\lambda_i(\cdot)$, and $d_i(\cdot)$, $i=1,2,\dots,p(\cdot)$, are all measurable and  $\theta$-invariant on $\widetilde{\Omega}$.
			\item[(iii)] Set $W_i(\omega)=\bigoplus_{j\geq i}V_j(\omega),\ i=1,2,\dots,p(\omega)$,  
			and $W_{p(\omega)+1}(\omega)=\{0\}$. Then it holds that
			\begin{equation}\label{conv: Lya-exp}
				\lim_{t\rightarrow\infty}\frac{1}{t}\log\norm{\Phi(t,\omega)x}=\lambda_i(\omega),\quad\forall\ x\in W_i(\omega)\backslash W_{i+1}(\omega),
			\end{equation}  
			for $i=1,2,\dots,p(\omega)$. The maps $V(\cdot)$ and $W(\cdot)$ from $\widetilde{\Omega}$ to the Grassmannian manifold are measurable. 
			\item[(iv)] It holds that
			\begin{equation*}
				W_i(\theta\omega)=A(\omega)W_i(\omega).
			\end{equation*}
			\item[(v)] When $(\Omega,\mathcal{F},\bP,\theta)$ is ergodic, i.e., every $B\in\mathcal{F}$ with $\theta^{-1}B=B$ satisfies $\bP(B)=0$ or $\bP(B)=1$, the functions  
			$p(\cdot)$, $\lambda_i(\cdot)$, and $d_i(\cdot)$, $i=1,2,\dots,p(\cdot)$, are constant on $\widetilde{\Omega}$.
		\end{itemize}
	\end{theorem}
	
	In Theorem~\ref{thm: MET}, $\lambda_1(\omega)>\lambda_2(\omega)>\cdots>\lambda_{p(\omega)}(\omega)$ are known as Lyapunov exponents and $\{0\}\subseteq W_{p(\omega)}(\omega)\subseteq \cdots\subseteq W_1(\omega)\subseteq \bR^d$ is the Oseledets filtration. We can see from the above theorem that both the Lyapunov exponents and the Oseledets filtration are  $A$-forward invariant. 
	
	The Lyapunov exponents describe the asymptotic growth rate of $\norm{\Phi(t,\omega)x}$ as $t\rightarrow\infty$. More specifically, \eqref{conv: Lya-exp} implies that when $x\in W_i(\omega)\backslash W_{i+1}(\omega)$, for any $\epsilon>0$, there exists some $T>0$, such that
	\begin{equation*}
		e^{t(\lambda_i(\omega)-\epsilon)}\leq\norm{\Phi(t,\omega)x}\leq e^{t(\lambda_i(\omega)+\epsilon)},
	\end{equation*}
	holds for any $t>T$. The subspaces spanned by eigenvectors of $\Lambda(\omega)$ corresponding to eigenvalues smaller than, equal to, and greater than $0$ are the stable subspace, center subspace, and unstable subspace, respectively. The stable and unstable subspaces correspond to exponential convergence and exponential divergence, respectively. When starting from the center subspace, we would get some sub-exponential behavior.
	
	The multiplicative ergodic theorem also generalizes to continuous time and two-sided time. We refer interested readers to \cite[Theorem 3.4.1, Theorem 3.4.11]{Arnold_RDS} for details.
	
	The stable, unstable, and center subspaces can be generalized to stable, unstable, and center manifolds when considering nonlinear systems, see e.g., \cite{Arnold_RDS, Ruelle-1979, Ruelle-82, Boxler-89-center, Weigu-Sternberg, Liu-06, Lian-10, Li-13, Guo-16}. These manifolds  play  similar roles in characterizing the local behavior of nonlinear random dynamical systems, as the subspaces for linear random dynamical systems. In particular, Hartman–Grobman theorem establishes the topological conjugacy between a nonlinear system and its linearization \cite{Wanner-95}. There are also other conjugacy results for random dynamical systems, see e.g., \cite{Weigu-Sternberg, Weigu-05, Weigu-08, Weigu-16}.

	\section{Main results}\label{sec: main result}
	\subsection{Setup of the random dynamical system} \label{sec: notation}
	
	Let us first specify the random dynamical system corresponding to the Algorithm~\ref{alg:random CGD}. 
	\begin{itemize}[wide]
		\item{\emph{Probability space}.} For each $t \in \bN$, denote  $(\Omega_t,\Sigma_t,\bP_t)$ the usual probability space for the distribution $\mathcal{U}_{\{1,2,\dots,n\}}\times \mathcal{U}_{[\alpha_{\min},\alpha_{\max}]}$, where $\mathcal{U}_{\{1,2,\dots,n\}}$ and $\mathcal{U}_{[\alpha_{\min},\alpha_{\max}]}$ are the uniform distributions on the set $\{1,2,\dots,n\}$ and interval $[\alpha_{\min},\alpha_{\max}]$ respectively. 
		Let $(\Omega,\mathcal{F},\bP)$ be the product probability space of all $(\Omega_t,\Sigma_t,\bP_t),\ t\in\bN$. Denote $\pi_t$ as the projection from $(\Omega,\mathcal{F},\bP)$ onto $(\Omega_t,\Sigma_t,\bP_t), \ t \in \bN$. Thus, a sample $\omega\in\Omega$ can be represented as a sequence  $((i_0,\alpha_0),(i_1,\alpha_1),\dots)$, where $(i_t,\alpha_t)=\pi_t(\omega),\ t\in\bN$. Let $\{\mathcal{F}_t\}_{t\in \bN}$ be the filtration defined by $$\mathcal{F}_t=\sigma\left\{(B_0\times\cdots\times B_t)\times \left(\prod_{j> t}\Omega_j\right):B_i\in \Sigma_i,\ i=0,1,\dots,t\right\}.$$
		
		\item{\emph{Metric dynamical system}.}
		The metric dynamical system on $\Omega$ is constructed by the (left) shifting operator $\tau:\Omega\rightarrow\Omega$ defined as 
		\begin{equation*}
			\tau(\omega)=\tau(\pi_0(\omega),\pi_1(\omega),\cdots) := (\pi_1(\omega),\pi_2(\omega),\cdots),
		\end{equation*}
		which is clearly measurable and $\bP$-preserving. The metric dynamical system is then given by $\theta(t)=\tau^t$ for $t\in\bN$. 
		
		\item{\emph{Random dynamical system}.}
		For any $\omega\in\Omega$ and $t\in\bN$, we define $\phi(\omega)$ to be a (nonlinear) map on $\bR^d$ as 
		\begin{equation*}
			\begin{split}
				\phi(\omega):\bR^d&\rightarrow\quad\quad\quad \bR^d\\
				x\ &\mapsto x-\alpha e_i e_i^{\top} \nabla f(x),
			\end{split}
		\end{equation*}
		where $(i,\alpha)=\pi_0(\omega)$ is the first pair/element in the sequence $\omega$, and we define the map $\varphi(t,\omega)$ via 
		\begin{equation*}
			\varphi(t,\omega)=\phi(\tau^{t-1}\omega)\circ\cdots \circ \phi(\tau\omega)\circ \phi(\omega), \qquad \text{for } t \geq 1,
		\end{equation*}
		while $\varphi(0, \omega)$ is the identity operator. It is clear that $\varphi(t,\omega)$ satisfies the cocycle property \eqref{cocycle} and hence defines a random dynamical system on $X=\bR^d$ over $\{\tau^t\}_{t\in\bN}$.
		The iterate of Algorithm~\ref{alg:random CGD} follows the random dynamical system as
		\begin{equation*}
			x_t=\phi(\tau^{t-1}\omega)x_{t-1}=\cdots=\phi(\tau^{t-1}\omega)\circ\cdots \circ \phi(\tau\omega)\circ \phi(\omega)x_0=\varphi(t,\omega)x_0.
		\end{equation*}
		It can be seen that $\{x_t\}_{t\in\bN}$ is $\{\mathcal{F}_t\}$-predictable, i.e., $x_t$ is $\mathcal{F}_{t-1}$-measurable for any $t\in\bNp$, since $x_t$ is determined by samples $(i_0,\alpha_0),(i_1,\alpha_1),\dots,(i_{t-1},\alpha_{t-1})$.
	\end{itemize}
	
	\smallskip 
	
	In our analysis, we will use linearization of the dynamical system $\varphi(t,\omega)$ at a critical point $x^*$ of $f$. Without loss of generality, we assume $x^*=0$; otherwise we consider the system with state being $x-x^*$. The resulting linear system, which depends on $H=\nabla^2 f(x^*)=(H_{ij})_{1\leq i,j\leq d}$, is given by (here and in the sequel, we use the superscript $H$ to indicate dependence on the matrix)
	\begin{equation}\label{linear: Phi}
		\Phi^H(t,\omega)=A^H(\tau^{t-1}\omega)\cdots A^H(\tau\omega)A^H(\omega),
	\end{equation}
	where 
	\begin{equation}\label{linear: A}
		A^H(\omega)=I-\alpha e_i e_i^{\top} H,\quad (i,\alpha)=\pi_0(\omega). 
	\end{equation}
	Note that $A^H(\cdot)$ is bounded in $\Omega$. 
	We know that $\left(\log\norm{A^H(\cdot)}\right)_+$ is integrable. When $\alpha< 1/|H_{ii}|$, 
	the matrix $A^H(\omega)=I-\alpha e_i e_i^{\top} H$ is invertable, and the inverse is given explicitly by applying the Sherman-Morrison formula:
	\begin{equation}\label{inverse-AH}
		\begin{split}
			A^H(\omega)^{-1}&=\left(I-\alpha e_i e_i^{\top} H\right)^{-1} = I + \frac{\alpha e_i e_i^{\top} H}{1 - \alpha H_{ii}}. 
		\end{split}
	\end{equation}
	In particular, we have 
	\begin{equation}
		\norm{A^H(\omega)^{-1}} \leq 1 + \frac{\alpha \norm{H}}{1 - \alpha |H_{ii}|}. 
	\end{equation}
	Thus, if we take the maximal stepsize $\alpha_{\max}$ such that $\alpha_{\max} < 1/\max_{1\leq i\leq d}|H_{ii}|$,  $\norm{A^H(\cdot)^{-1}}$ is bounded in $\Omega$, and as a result $\left(\log\norm{A^H(\cdot)^{-1}}\right)_+$ is also integrable. 
	Therefore, the assumptions of Theorem~\ref{thm: MET} hold. The shifting operator $\tau$ is ergodic on $(\Omega,\mathcal{F},\bP)$ by Kolmogorov's $0$--$1$ law. Then Theorem~\ref{thm: MET} applies for $\theta=\tau$ with $p^H(\cdot)$, $\lambda_i^H(\cdot)$, and $d_i^H(\cdot)$ all being a.e. constant. For any $\omega\in\widetilde{\Omega}$ that is the set in Theorem~\ref{thm: MET} satisfying $\bP(\widetilde{\Omega})=1$, 
	we denote
	\begin{equation}\label{W_pm}
		W_+^H(\omega)=\bigoplus_{\lambda_i>0}V_i^H(\omega),\ \text{and}\ W_-^H(\omega)=\bigoplus_{\lambda_i\leq 0}V_i^H(\omega).
	\end{equation}
	Then the following invariant property holds 
	\begin{equation*}
		W_-^H(\tau\omega)=A^H(\omega)W_-^H(\omega).
	\end{equation*}
	Note that $W_-^H(\omega)$ works as a center-stable subspace. That is, for any $x\in W_-^H(\omega)$ and any $\epsilon>0$, it holds that $\norm{\Phi(t,\omega)x}\leq e^{t\epsilon}$, for sufficiently large $t$, and for $x\notin W_-^H(\omega)$, $\norm{\Phi(t,\omega)x}$ grows exponentially as $t\rightarrow\infty$ with rate greater than $\min_{\lambda_i>0}\lambda_i -\epsilon$ for any $\epsilon>0$.

	\subsection{Assumptions}
	
	In this section, we specify the assumptions of the objective function $f$ in this paper. 
	The first is a standard smoothness assumption of $f$. 
	\begin{assumption}\label{Assump: Hess_bdd} $f \in C^2(\bR^d)$ and the Hessian $\nabla^2 f$ is uniformly bounded, i.e., there exists $M>0$ such that $\norm{\nabla^2 f(x)}\leq M$ for all $x\in\bR^d$.
	\end{assumption}
	
	An optimization algorithm is expected to converge, under some reasonable assumptions, to a critical point of $f$ where the gradient vanishes. Our aim is to further characterize the possible limits of the algorithm iterates. For this purpose, we distinguish $\calA$, the set of all strict saddle points (including local maxima with non-degenerate Hessian) of $f$:
	\begin{equation*}
		\calA := \{x\in\bR^d:\nabla f(x)=0,\
		\lambda_{\min}(\nabla^2 f(x))<0\},
	\end{equation*}
	where we use the subscript $s$ to emphasize that it is strict. 
	Due to the presence of negative eigenvalue of Hessian, if we were considering the gradient dynamics near the critical point, the saddle point would be an unstable equilibrium. Our first result is that this instability also occurs in the linear random dynamical system $\Phi^H(t,\omega)$ where $H=\nabla^2 f(x^*)$. In other words, the dimension of $W^H_+(\omega)$ defined in \eqref{W_pm} is greater than $0$. 
	While this would mainly serve as a preliminary step for our analysis of the nonlinear dynamics, the conclusion by itself might be of interest, and is stated as follows. The proof will be deferred to Section~\ref{sec: positive-Lya}. 
	
	\begin{proposition}\label{prop: posi-lamb1}
		%For any $x^*\in\calA$, if 
		Let $H$ have a negative eigenvalue and
		$0<\alpha_{\min}<\alpha_{\max}<1/\max_{1\leq i\leq d}|H_{ii}|$, then the largest Lyapunov exponent of $\Phi^H(t,\omega)$ is positive. 
	\end{proposition}
	
	Our goal is to generalize such results to the nonlinear dynamics near strict saddle points of $f$, for which, we would require two additional assumptions, as the following. 
	\begin{assumption}\label{Assump: non-degenerate}
		For every $x^*\in \calA$, $\nabla^2 f(x^*)$ is non-degenerate, i.e., $x^*$ is a non-degenerate critical point of $f$ in the sense that any eigenvalue of $\nabla^2 f(x^*)$ is nonzero.
	\end{assumption}
	Assumption~\ref{Assump: non-degenerate} is also a standard assumption, which in particular  guarantees that each strict saddle point is isolated, due to the non-degenerate Hessian. For each strict saddle point, Proposition~\ref{prop: posi-lamb1} guarantees that the corresponding unstable subspace $W^H_+(\omega)$ is non-trivial (has dimension at least $1$). We would in fact require a stronger technical assumption on its structure. 
	\begin{assumption}\label{Assump: non-zero-proj}
		For every $x^*\in \calA$, it holds that $\mathcal{P}_+^H(\omega) e_i\neq 0$, for every $i\in\{1,2,\dots,d\}$ and almost every $\omega\in\Omega$, where $\mathcal{P}_+^H(\omega)$ is the orthogonal projection onto $W_+^H(\omega)$ with $H=\nabla^2 f(x^*)$.
	\end{assumption}
	We expect that Assumption~\ref{Assump: non-zero-proj} holds generically. We also show in Appendix~\ref{sec: ex-asp3}, Assumption~\ref{Assump: non-zero-proj} can be verified when $H$ has no zero off-diagonal elements (and $W^H_+(\omega)$ is not trivial). However, there exist cases that Assumption~\ref{Assump: non-zero-proj} might not hold. One example is $H=\nabla^2 f(x^*)=\text{diag}(H_1,H_2)$ where $H_1\in\bR^{d_1\times d_1}$ only has positive eigenvalues and $H_2\in\bR^{d_2\times d_2}$ only has negative eigenvalues, which implies that $W_-^H(\omega)=\text{span}\{e_i:1\leq i\leq d_1\}$ and $W_+^H(\omega)=\text{span}\{e_i:d_1+1\leq i\leq d\}$.
	
	We also remark that Assumption~\ref{Assump: Hess_bdd} is essential in our framework, since the linearized system is defined using the Hessian matrix. Analysis of randomized coordinate method for non-smooth optimization problems requires new techniques and deserves future research.

	\subsection{Main results}
	
	Given an initial guess $x_0$, the behavior of the algorithm, in particular, the limit of $x_t$, depends on the particular sample $\omega \in \Omega$.
	For any $x^*\in\calA$, we denote the set of all $\omega$ such that the algorithm starting at $x_0$ would converge to $x^*$:
	\begin{equation*}
		\calC(x^*,x_0) :=\left\{\omega\in\Omega:\lim_{t\rightarrow\infty}x_t=\lim_{t\rightarrow\infty}\varphi(t,\omega)x_0=x^*\right\}.
	\end{equation*}
	We further define the set $\calC(\calA,x_0)$ as the union of all $\calC(x^*,x_0)$ over $x^*\in\calA$,
	\begin{equation*}
		\calC(\calA,x_0):=\bigcup_{x^*\in\calA}\calC(x^*,x_0).
	\end{equation*}
	Thus, if $\omega \notin \calC(\calA,x_0)$, the limit $\lim_{t \to \infty} x_t$, if exists, will not be one of the strict saddle points. Our main result in this paper proves that the set is of measure zero, i.e., for any initial guess $x_0$ that is not a strict saddle point, with probability $1$, Algorithm~\ref{alg:random CGD} will not converge to a strict saddle point.
	\begin{maintheorem}\label{thm: main}
		Suppose that Assumptions~\ref{Assump: Hess_bdd}, \ref{Assump: non-degenerate}, and \ref{Assump: non-zero-proj} hold and that $0 < \alpha_{\min} < \alpha_{\max} < 1/M$, then  
		for any $x_0\in \bR^d\backslash\calA$, it holds that
		\begin{equation*}
			\bP(\calC(\calA,x_0))=0.
		\end{equation*}
	\end{maintheorem}
	
	The intuition behind the proof of Theorem~\ref{thm: main} is to compare the nonlinear dynamics around a strict saddle point $x^* \in \calA$ with its linearization, as the linear dynamics has non-trivial unstable subspace, thanks to Proposition~\ref{prop: posi-lamb1}. Ideally, one would hope that the nonlinear dynamics would closely follow the linear dynamics, and thus leave the neighborhood of $x^*$ eventually; the obstacle for such argument is however that the approximation of the linearization is only valid for a finite time interval. Therefore, to establish the instability behavior of the nonlinear dynamics, we would need a much more refined and quantitative argument using the instability of the linear system. In particular, we would need to show that over a finite interval, with high probability, the linear system, and hence the nonlinear system, would drive $x_t$ away from the strict saddle point with quantitative bounds; see Theorem~\ref{thm: local-growth} in Section~\ref{sec: finitecomp}.
	Theorem~\ref{thm: main} then follows from an argument with a similar spirit as law of large numbers; see Section~\ref{sec: proofmain}.
	
	\begin{remark}\label{random-alpha}
		The technical Assumption~\ref{Assump: non-zero-proj} and the randomness in stepsizes are made so that the iterate $x_t = x_{t-1} - \alpha_{t-1} e_{i_{t - 1}} e_{i_{t-1}}^{\top} \nabla f(x_{t-1})$ would obtain some non-trivial component in the unstable subspace, which would be further amplified within a sufficiently long but finite time interval. When $\norm{\mathcal{P}_+^H(\tau^t \omega) e_{i_{t - 1}}}$ and $\lvert e_{i_{t-1}}^{\top} \nabla f(x_{t-1})\rvert$ are fixed and relatively large, a random $\alpha_{t-1}$ would keep  $\norm{\mathcal{P}_+^H(\tau^t \omega) x_t}$ away from $0$ with high probability; see Section~\ref{sec: finitecomp} for more details. It is an interesting open question whether it is possible to establish similar results without such assumptions. Our conjecture is that $\bP(\Omega_s(x_0))=0$ still holds for $x_0\in \bR^d\backslash\calA$ unless $x_0$ is located in a set with Lebesgue measure zero, similar to the results established in \cite{Lee-2019}.
	\end{remark}

	\smallskip 
	
	As an application of our main result Theorem~\ref{thm: main}, we can obtain the global convergence to stationary points with no negative Hessian eigenvalues for Algorithm~\ref{alg:random CGD}.
	More specifically, denote by
	\begin{equation*}
		\crit :=\{x\in\bR^d:\nabla f(x)=0\},
	\end{equation*}
	the set of all critical points of $f$, we have the following corollary, which will also be proved in Section~\ref{sec: proofmain}.
	\begin{corollary}\label{cor: global-conv}
		Under the same assumptions as in Theorem~\ref{thm: main} and assume further that every $x^*\in\crit$ is an isolated critical point. For any $x_0\in\bR^d\backslash \calA$ with bounded level set $L(x_0)=\{x\in\bR^d:f(x)\leq f(x_0)\}$, with probability $1$, $\{x_t\}_{t\in\bN}$ is convergent with limit in $\crit\backslash\calA$.
	\end{corollary}
	
	\begin{remark}
		In Corollary~\ref{cor: global-conv}, if we further assume that all saddle points of $f$ are strict, then the algorithm iterate converges to a local minimum with probability $1$. Let us also mention that for many non-convex problems, saddle points are suboptimal while there do not exist ``bad" local minima, e.g. phase retrieval \cite{sun2018geometric}, deep learning \cite{kawaguchi2016deep, lu2017depth}, and low-rank matrix problems \cite{ge2017no}. For these problems, convergence to local minima suffice to guarantee good performance. 
	\end{remark}

	\begin{remark}
		In our setting without adding noise to gradient or iterate, we cannot hope for good convergence rates for arbitrary initial iterate. In fact as shown in \cite{Du-17}, the convergence of deterministic gradient descent algorithm to a local minimum might take exponentially long time; we expect similar behavior for the randomized coordinate gradient descent Algorithm~\ref{alg:random CGD}. Let us also remark that while we need the random stepsize as discussed in Remark~\ref{random-alpha}, the interval $[\alpha_{\min}, \alpha_{\max}]$ could be made arbitrarily small; the result holds as long as $0<\alpha_{\min} < \alpha_{\max}<1/M$. 
	\end{remark}
	
	\section{Proofs}\label{sec: proof}
	
	We collect all proofs in this section. 
	
	\subsection{Analysis of the linearized system}\label{sec: positive-Lya}
	We will first study the linear dynamical system, for which we assume the objective function is given by 
	\begin{equation}\label{obj: quadratic}
		f^H(x)=\frac{1}{2}x^{\top} H x,
	\end{equation}
	where $H$ is a symmetric matrix in $\bR^{d\times d}$ with at least one negative eigenvalue. In this case, the coordinate descent algorithm is given by 
	\begin{equation*}
		x_{t+1}=\left(I-\alpha_t e_{i_t} e_{i_t}^{\top} H\right)x_t,
	\end{equation*}
	which corresponds to the linear dynamical system  $\Phi^H(t,\omega)$ with single step map $A^H(\omega)$,  defined in \eqref{linear: Phi} and \eqref{linear: A} respectively. 
	
	Our main goal in this subsection is to prove Proposition~\ref{prop: posi-lamb1} for this linear dynamical system which states that at least one Lyapunov exponent of $\Phi^H(t,\omega)$ is positive. It suffices to show that there exists some $x_0$, such that $\norm{x_t}$ grows exponentially to infinity, which will follow from an energy argument, similar to the proof of \cite[Proposition 5]{Lee-2019}. Although we consider a randomized coordinate gradient descent algorithm instead of a cyclic one, one step, i.e., Lemma~\ref{lem: linear_decay}, in the proof of Proposition~\ref{prop: linear_decay} follows closely the proof in \cite[Appendix A]{Lee-2019}.  We start from $x_0$ with $f^H(x_0)<0$ and consider a finite time interval with length $m\geq d$. Proposition~\ref{prop: linear_decay} establishes a quantitative decay estimate for $f^H(x_{t+m})$ compared with $f^H(x_t)$, which leads to our desired  result Proposition~\ref{prop: posi-lamb1}. 
	
	\begin{proposition}\label{prop: linear_decay}
		Let $m\geq d$ be fixed. For the objective function \eqref{obj: quadratic} with $\lambda_{\min}(H)<0$, suppose 
		that $0<\alpha_{\min}<\alpha_{\max}<1/\max_{1\leq i\leq d}|H_{ii}|$, there exists $c\in(0,1)$ depending on $m$, $H$, $\alpha_{\min}$, and $\alpha_{\max}$, 
		such that 
		\begin{equation*}
			f^H(x_{t+m})-f^H(x_t)\leq c\, f^H(x_t),
		\end{equation*}
		holds as long as $f^H(x_t)<0$ and $\{1,2,\dots,d\}=\{i_t,i_{t+1},\dots,i_{t+m-1}\}$ (in the sense of sets). 
	\end{proposition}
	
	\begin{remark}
		The condition $\{1,2,\dots,d\}=\{i_t,i_{t+1},\dots,i_{t+m-1}\}$ above is known as ``generalized Gauss-Seidel rule'' in the literature of coordinate methods \cite{Tseng-09, Wright-12}.
	\end{remark}

	\begin{proof} 
		Without loss of generality, we assume that $t=0$. 
		Due to the choice of $\alpha_{\max}$, we have the following simple non-increasing property for any $t'\in\bN$:
		\begin{equation}\label{decay: quad}
			\begin{split}
				f^H(x_{t'+1})&=\frac{1}{2}x_{t'+1}^{\top} H x_{t'+1}\\
				&=\frac{1}{2}x_{t'}^{\top} \left(I-\alpha_{t'} e_{i_{t'}} e_{i_{t'}}^{\top} H\right)^{\top} H \left(I-\alpha_{t'} e_{i_{t'}} e_{i_{t'}}^{\top} H\right) x_{t'}\\
				&=f^H(x_{t'})-\alpha_{t'} \left(e_{i_{t'}}^{\top} H x_{t'}\right)^2+\frac{1}{2}\alpha_t^2 e_{i_{t'}}^{\top} H e_{i_{t'}} \left(e_{i_{t'}}^{\top} H x_{t'}\right)^2\\
				&\leq f^H(x_{t'})-\frac{\alpha_{t'}}{2} \left(e_{i_{t'}}^{\top} H x_{t'}\right)^2.
			\end{split}
		\end{equation}
		
		Write $x_0=y^*+y_0$ with $y^*\in\ker(H)$ and $y_0\in\ran(H)$. Let 
		\begin{equation}
			\label{def: y}
			y_{t'+1}=y_{t'}-\alpha_{t'} e_{i_{t'}} e_{i_{t'}}^{\top} H y_{t'}, \qquad t'=0,1,\dots,m-1.
		\end{equation} 
		Then $x_{t'}=y^*+y_{t'}$ holds for any $t'=0,1,\dots,m$. Using \eqref{decay: quad}, to give an upper bound for $f^H(x_{t+m})-f^H(x_t)$, we would like a nontrivial lower bound for $\alpha_{t'}\bigl(e_{i_{t'}}^{\top} H x_{t'}\bigr)^2=\alpha_{t'}\bigl(e_{i_{t'}}^{\top} H y_{t'}\bigr)^2$ for some $t'\in\{t,t+1,\dots,t+m-1\}$, which is guaranteed by Lemma~\ref{lem: linear_decay}, whose proof will be postponed. 
		
		\begin{lemma} \label{lem: linear_decay}
			Suppose that $\{1,2,\dots,d\}=\{i_0,i_1,\dots,i_{m-1}\}$. For any
			\begin{equation}\label{def:delta}
				0<\delta\leq\min\left\{\frac{1}{2m},\frac{\alpha_{\min}\sigma_{\min}(H)}{2\sqrt{m}(m\alpha_{\max}\sigma_{\max}(H)+1)}\right\},
			\end{equation}
			where $\sigma_{\min}(H)$ and $\sigma_{\max}(H)$ are the smallest and largest positive singular values of $H$ respectively, 
			if $0\neq y_0\in \ran(H)$, then there exists $T\in \{0,1,\dots,m-1\}$ such that $\alpha_{T}\left\vert e_{i_{T}}^{\top} H y_{T}\right\vert\geq \delta\norm{y_{T}}$, where the sequence $y_t$ is given as in \eqref{def: y}. 
		\end{lemma}
		
		Assuming the lemma, there exists $T\in\{0,1,\dots,m-1\}$ that 
		\begin{equation*}
			\alpha_T\left\vert e_{i_{T}}^{\top} H y_{T}\right\vert\geq \delta\norm{y_{T}},
		\end{equation*}
		with a fixed $\delta>0$ satisfying \eqref{def:delta}. We can further constrain that $\frac{\delta^2}{\alpha_{\min}\sigma_{\max}(H)}<1$. Thus, we have
		\begin{equation*}\label{decay esti}
			\begin{split}
				f^H(x_m)&\leq f^H (x_{T+1})\leq f^H(x_T)- \frac{\alpha_T}{2}\left(e_{i_{T}}^{\top} H x_{T}\right)^2\leq f^H (x_{T}) -\frac{\delta^2}{2\alpha_T}\norm{y_{T}}^2,
			\end{split}
		\end{equation*}
		which combined with $\sigma_{\max}(H)\norm{y_{T}}^2\geq -y_{T}^{\top} H y_{T}=-x_{T}^{\top} H x_{T}=-2f^H(x_T)$,
		yields that 
		\begin{equation*}
			f^H (x_m) \leq\left( 1 +\frac{\delta^2}{\alpha_T \sigma_{\max}(H)}\right)f^H (x_{T}) \leq\left( 1 +\frac{\delta^2}{\alpha_T \sigma_{\max}(H)}\right)f^H (x_0). 
		\end{equation*}
		Set $c=\frac{\delta^2}{\alpha_{\max} \sigma_{\max}(H)}$ and we get that $f^H(x_m)-f^H(x_0)\leq c f^H(x_0)$.
	\end{proof}

	We finish the proof by establishing Lemma~\ref{lem: linear_decay} below. 
	
	\begin{proof}[Proof of Lemma~\ref{lem: linear_decay}] Suppose on the contrary that $\alpha_t\left\vert e_{i_t}^{\top} H y_t\right\vert< \delta\norm{y_t}$ for any $t\in \{0,1,\dots,m-1\}$. It holds that
		\begin{equation*}
			\norm{y_1-y_0}=\alpha_0\left\vert e_{i_0}^{\top} H y_0\right\vert<\delta\norm{y_0}<2\delta\norm{y_0}.
		\end{equation*}
		We claim that 
		\begin{equation} \label{eq: disty}
			\norm{y_t-y_0}<2 t\delta\norm{y_0},
		\end{equation}
		for any $t=1,2,\dots,m-1$. By induction, assume that  $\norm{y_t-y_0}<2t\delta\norm{y_0}$ holds for some $t\in\{1,2,\dots,m-2\}$, then 
		\begin{equation*}
			\norm{y_{t+1}-y_t}=\alpha_t\left\vert e_{i_t}^{\top} H y_t\right\vert< \delta\norm{y_t}<\delta(2t\delta+1)\norm{y_0}<2\delta\norm{y_0},
		\end{equation*}
		where the last inequality uses $2 t \delta < 2m \delta \leq 1 $. It follows $\norm{y_{t+1}-y_0}\leq \norm{y_t-y_0}+\norm{y_{t+1}-y_t}<2(t+1)\delta\norm{y_0}$.
		
		Using \eqref{eq: disty} and $\max_{1\leq i\leq d} \norm{H e_i}\leq \sigma_{\max}(H)$, we have 
		\begin{equation*}
			\begin{split}
				\alpha_t\bigl\lvert e_{i_t}^{\top} H y_0\bigr\rvert&\leq \alpha_t\bigl\lvert e_{i_t}^{\top} H (y_t-y_0)\bigr\rvert+\alpha_t\bigl\lvert e_{i_t}^{\top} H y_t\bigr\rvert\\
				&<\alpha_{\max}\sigma_{\max}(H)\cdot 2 t\delta\norm{y_0}+2\delta\norm{y_0}\\
				&<2\delta(m\alpha_{\max}\sigma_{\max}(H)+1)\norm{y_0},
			\end{split}
		\end{equation*}
		for $t=0,1,\dots,m-1$. Since $\text{span}\{e_{i_k}:k=0,1,\dots,m-1\}=\bR^d$, noticing that $y_0\in\ran{H}$, we have
		\begin{equation*}
			\begin{split}
				\alpha_{\min}\sigma_{\min}(H) \norm{y_0}&\leq \alpha_{\min}\norm{H y_0}\leq\left(\sum_{t=0}^{m-1} \left(\alpha_t\bigl\lvert e_{i_t}^{\top} H y_0\bigr\rvert\right)^2\right)^{1/2}\\
				& <2\delta\sqrt{m}(m\alpha_{\max}\sigma_{\max}(H)+1)\norm{y_0},
			\end{split}
		\end{equation*}
		which contradicts with the choice of $\delta$ in \eqref{def:delta}.
	\end{proof}

	We are now ready to prove Proposition~\ref{prop: posi-lamb1} which states the existence of positive Lyapunov exponent of the linear dynamical system. 
	
	\begin{proof}[Proof of Proposition~\ref{prop: posi-lamb1}] It suffices to show that for almost every $\omega\in\Omega$ there exist some $x_0\in \bR^d$, $\epsilon>0$, and $T>0$, such that $x_t=\Phi(t,\omega)x_0$ satisfies $\norm{x_t}\geq e^{\epsilon t}$ for any $t>T$. Let $x_0$ be an eigenvector corresponding to a negative eigenvalue of $H$. Then it holds that $f^H(x_0)< 0$. Consider a fixed $m\geq d$. For any $k\in\bN$, set $I_k=1$ if $\{1,2,\dots,d\}=\{i_{km},i_{km+1},\dots,i_{km+m-1}\}$ 
		and $I_k=0$ otherwise. We can see that the random variables, $I_0,I_1,I_2,\dots$, are independent and identically distributed with $\bE I_0=\bP(I_0=1)\in(0,1)$. By Proposition~\ref{prop: linear_decay}, we obtain that
		\begin{equation*}
			f^H(x_{(k+1)m})\leq\begin{cases}
				(1+c)f^H(x_{km}),&\text{if }I_k=1,\\
				f^H(x_{km}),&\text{if }I_k=0,
			\end{cases}
		\end{equation*}	
		where $c$ is the constant from Proposition~\ref{prop: linear_decay}.
		Therefore, %\jl{missing $1/2$}
		\begin{equation*}
			\frac{\lambda_{\min}(H)}{2}\norm{x_{km}}^2\leq f^H(x_{km})\leq(1+c)^{\sum_{j=0}^{k-1} I_j}f^H(x_0),
		\end{equation*}
		which implies that
		\begin{equation}\label{norm:xkm}
			\norm{x_{km}}\geq \left(\frac{2f^H(x_0)}{\lambda_{\min}(H)}\right)^{1/2}\cdot (1+c)^{\frac{1}{2}\sum_{j=0}^{k-1} I_j}.
		\end{equation}
		Note that $\bE |I_0|=\bE I_0<\infty$. The strong law of large number suggests that for almost every $\omega\in\Omega$, there exists some $K$, such that for all $k \geq K$
		\begin{equation}\label{SLLN}
			\sum_{j=0}^{k-1} I_j\geq \frac{\bE I_0}{2}k.
		\end{equation}
		Combining \eqref{norm:xkm} and \eqref{SLLN}, we arrive at
		\begin{equation*}
			\norm{x_{km}}\geq \left(\frac{2f^H(x_0)}{\lambda_{\min}(H)}\right)^{1/2}\cdot (1+c)^{\frac{\bE I_0}{4m}\cdot km}.
		\end{equation*}
		Noticing that $(1+c)^{\frac{\bE I_0}{4m}}$ is greater than $1$, $\norm{x_{km}}$ grows exponentially in $km$ and we complete the proof.
	\end{proof}
	
	\subsection{Finite block analysis}\label{sec: finitecomp}
	
	In this subsection, we study the behavior of the nonlinear dynamical system near a strict saddle point of $f$, which without loss of generality can be assumed to be $x^*=0$. As mentioned above, in a small neighborhood of $x^*$, while it is not possible to control the difference between nonlinear and linear systems for infinite time,  the nonlinear system can be approximated by the linear system during a finite time horizon. 
	
	The main conclusion of this subsection is the following theorem which states that after a finite time interval with length $T$, the distance of the iterate from $x^*=0$ will be amplified exponentially with high probability.
	\begin{theorem}\label{thm: local-growth}
		Suppose that Assumptions~\ref{Assump: Hess_bdd}, \ref{Assump: non-degenerate}, and \ref{Assump: non-zero-proj} hold and that $0<\alpha_{\min}<\alpha_{\max}<1/M$. There exists $\epsilon_*\in(0,1/6)$ such that for any $\epsilon\in(0,\epsilon_*)$, we have $T_* = T_*(\epsilon) \in\bNp$ that for any $T \in \bNp$ with $T\geq T_*$ and any $t\in\bN$, conditioned on $\mathcal{F}_{t-1}$, with probability at least $1-4\epsilon$, it holds for all $x_t \in V$ that
		\begin{equation}\label{local-growth}
			\norm{x_{t+T}}\geq \exp\biggl(\frac{6\epsilon}{1 - 6 \epsilon} \bigl\lvert \log(1-M\alpha_{\max}) \bigr\rvert \, T\biggr) \norm{x_t},
		\end{equation}
		where $V$ is a neighborhood of $x^* = 0$, depending on $\epsilon$, $T$, and $f$ near $x^*$. 
	\end{theorem}
	
	The lower bound \eqref{local-growth} quantifies the amplification of $\norm{x_{t+T}}$: While we always have $\norm{x_{t+T}}\geq (1-M\alpha_{\max})^{T}\norm{x_t}$ (see \eqref{r-minus} below), the result states that with probability at least $1 - 4\epsilon$, the amplification factor is at least the right-hand side of \eqref{local-growth}, which is exponentially large in $T$. Hence on average $\norm{x_{t+T}}$ would be much larger than $\norm{x_t}$. This would lead to escaping of the iterate from the neighborhood of $x^* = 0$.
	
	To prove Theorem~\ref{thm: local-growth}, we would require a more quantitative characterization of the behavior of its linearization at $x^*$. In particular, we need a high probability estimate of the distance of the iterate from $x^*$ after some time interval. For this purpose, conditioned on $\mathcal{F}_{t-1}$ with the iterate $x_t$, we will first show in Lemma~\ref{lem: nontrial-proj} that, after some finite time, the orthogonal projection of the iterate $x_{\varrho_t}$ onto the unstable subspace, where $t < \varrho_t \leq t + L$ for some constant $L$, is significant. The component in the unstable subspace would then be further amplified subsequently by $\Phi^H(S, \tau^{\varrho_t} \omega)$, where $H=\nabla^2 f (x^*)$. Here the time duration $S$ 
	would be chosen sufficiently large such that the distance from $x_{\varrho_t + S}$ to $x^*$ is exponentially amplified. Theorem~\ref{thm: local-growth} follows by setting $T = L + S$. 
	In the second step above, we would need to control the closeness between the linear and nonlinear systems within a time horizon with length $S$. 
	
	Such a finite block analysis approach has been used to establish the stability of Lyapunov exponent of random dynamical systems \cite{Ledrappier-91, Froyland-15}, which inspired our proof technique for Theorem~\ref{thm: local-growth} and Theorem~\ref{thm: main}.
	
	\medskip

	We first set the small constant $\epsilon$ in Theorem~\ref{thm: local-growth} which controls the failure probability of the amplification bound. Let $\lambda_1>\lambda_2>\cdots>\lambda_p$ be the Lyapunov exponents of the linearized system at $x^*=0$. We set
	\begin{equation}\label{lambplus-gamma}
		\lambda_+=\min_{\lambda_i>0}\lambda_i\quad\text{and}\quad 0<\gamma<\frac{1}{2}\min\left\{\min_{1\leq i< p}|\lambda_i-\lambda_{i+1}|,\lambda_+\right\}. 
	\end{equation}
	Note that $\gamma<\lambda_+$.
	Let $\epsilon_*\in(0,1/6)$ be sufficiently small such that
	\begin{equation}\label{def: eps}
		(1-6\epsilon)(\lambda_+-\gamma)+6\epsilon \cdot \log(1-M\alpha_{\max})>0,\quad \forall\ \epsilon\in(0,\epsilon_*). 
	\end{equation}
	The reason for such choice will become clear later (see \eqref{large S2}). For the rest of the section, we will consider a fixed $\epsilon\in(0,\epsilon_*)$. 
	
	We now state and prove several lemmas for Theorem~\ref{thm: local-growth}. First, in the following Lemma~\ref{lem: L}, we construct a stopping time $\varrho_t-1$ that is bounded almost surely and  that the component of the gradient $\bigl\lvert e_{i_{\varrho_t-1}}^{\top} \nabla f(x_{\varrho_t-1})\bigr\rvert$ is comparable with 
	$\norm{\nabla f(x_{\varrho_t-1})}$ in amplitude 
	with high probability. 
	
	\begin{lemma}\label{lem: L}
		Let $0<\mu\leq\frac{1}{\sqrt{d}}$ be a fixed constant. 
		There exists some constant $L>0$, such that for any $t\in\bN$, there exists a measurable $\varrho_t:\Omega\rightarrow \bNp$ such that $t < \varrho_t \leq t + L$ and
		\begin{equation}\label{eq: rho_t}
			\bP\Bigl(\bigl\lvert e_{i_{\varrho_t-1}}^{\top} \nabla f(x_{\varrho_t-1})\bigr\rvert \geq \mu\norm{\nabla f(x_{\varrho_t-1})}\;\Big\vert\; \mathcal{F}_{t-1}\Bigr)\geq 1-\epsilon.
		\end{equation}
	\end{lemma}
	
	\begin{proof}
		For any $t\in\bN$, use $\ell_0$ to denote the smallest non-negative integer $\ell$, such that
		\begin{equation*}
			\ell_0 = \arg\min_{\ell} \Bigl\{ \ell \in \mathbb{N}_+ \,:\,	\bigl\lvert e_{i_{t+\ell-1}}^{\top} \nabla f(x_{t+\ell-1})\bigr\rvert\geq \mu\norm{\nabla f(x_{t+\ell-1})} \Bigr\}.
		\end{equation*}
		It is clear that $\bP(\ell_0 >\ell\mid\mathcal{F}_{t-1})\leq (1-1/d)^\ell$, since for each step the coordinate is randomly chosen. Hence, there exists some $L>0$, such that 
		\begin{equation*}
			\bP(\ell_0 \leq L\mid\mathcal{F}_{t-1})\geq 1-\epsilon.
		\end{equation*}
		We finish the proof by setting $\varrho_t = t + \min\{ \ell_0, L \}$
		which has the desired property. 
	\end{proof}
	
	\medskip
	
	We now carry out the amplification part of the finite block analysis 
	for the linearized dynamics at $x^* = 0$. To simplify expressions in the following, for $t_1<t_2$, we introduce the short-hand notation
	\begin{equation*}
		(i,\alpha)_{t_1:t_2-1}=\big((i_{t_1},\alpha_{t_1}),\dots,(i_{t_2-1},\alpha_{t_2-1})\big)\in\Omega_{t_1}\times\cdots\times\Omega_{t_2-1},
	\end{equation*}
	and the finite time transition matrix (i.e., composition of linear maps)
	\begin{equation*}
		\Phi^H\bigl((i,\alpha)_{t_1:t_2-1}\bigr)=(I-\alpha_{t_2-1}e_{i_{t_2-1}}e_{i_{t_2-1}}^{\top} H)\cdots (I-\alpha_{t_1} e_{i_{t_1}}e_{i_{t_1}}^{\top} H).
	\end{equation*}
	Recall that $(\Omega_t,\Sigma_t,\bP_t)$ is the probability space for $\mathcal{U}_{\{1,2,\dots,d\}}\times \mathcal{U}_{[\alpha_{\min},\alpha_{\max}]}$ for $t\in\bN$. We also denote $\mathcal{P}_+^H\bigl((i,\alpha)_{t_1:t_2-1}\bigr)$ as the projection operator onto the subspace spanned by the right singular vectors of $\Phi^H\bigl((i,\alpha)_{t_1:t_2-1}\bigr)$ corresponding to $d_+$ largest singular values, where $d_+=\sum_{\lambda_i>0}d_i$ and $d_i$ is the dimension of the $i$-th eigenspace as in Theorem~\ref{thm: MET} (ii) for the linearized system at $x^*$.

	As we mentioned in the proof sketch, we want $\Phi^H(S,\tau^{\varrho_t}\omega)=\Phi^H\left((i,\alpha)_{\varrho_t:\varrho_t+S-1}\right)$ to amplify $x_{\varrho_t}$, for which we need to establish a non-trivial lower-bound for the unstable component $\norm{\mathcal{P}_+^H\left((i,\alpha)_{\varrho_t:\varrho_t+S-1}\right)x_{\varrho_t}}$. 
	This is achieved by several lemmas. We will establish three lower bound in the sequel: 
	\begin{itemize}
		\item $\norm{\mathcal{P}_+^H\left(\tau^{\varrho_t}\omega\right)e_{i_{\varrho_t-1}}}$ using Lemma~\ref{lem: Omega1};
		\item $\norm{\mathcal{P}_+^H\left((i,\alpha)_{\varrho_t:\varrho_t+S-1}\right)e_{i_{\varrho_t-1}}}$ in Lemma~\ref{lem: T-Omega2}; and finally the desired
		\item $\norm{\mathcal{P}_+^H\left((i,\alpha)_{\varrho_t:\varrho_t+S-1}\right)x_{\varrho_t}}$ in Lemma~\ref{lem: nontrial-proj}.
	\end{itemize} 
	\smallskip 
	Let us first control $\norm{\mathcal{P}_+^H\left(\tau^{\varrho_t}\omega\right)e_{i_{\varrho_t-1}}}$ in the following lemma, which utilizes Assumption~\ref{Assump: non-zero-proj}.
	For simplicity of notation, in Lemma~\ref{lem: Omega1} and Lemma~\ref{lem: T-Omega2}, we state the results for $\norm{\mathcal{P}_+^H(\omega)e_i}$ and $\norm{\mathcal{P}_+^H((i,\alpha)_{0:S-1})e_j}$ instead, which is slightly more general.

	\begin{lemma}\label{lem: Omega1}
		Under Assumption~\ref{Assump: non-zero-proj}, 
		there exist $\delta>0$ and measurable $\Omega_1^\epsilon\subset\widetilde{\Omega}$, where $\widetilde{\Omega}$ is from Theorem~\ref{thm: MET}, 
		such that $\bP(\Omega_1^\epsilon)\geq 1-\epsilon$ 
		and 
		\begin{equation*}
			\norm{\mathcal{P}^H_+(\omega) e_i}\geq\delta,\quad \forall\ \omega\in\Omega_1^\epsilon,\ i\in\{1,2,\dots,d\}.
		\end{equation*}
	\end{lemma}
	
	\begin{proof}
		Assumption~\ref{Assump: non-zero-proj} implies that 
		\begin{equation*}
			\bP\bigl(\{\omega\in\widetilde{\Omega}:\norm{\mathcal{P}^H_+(\omega) e_i}>0,\quad \forall\ i\in\{1,2,\dots,d\}\}\bigr) = 1
		\end{equation*}
		Notice that 
		\begin{multline*}
			\bigl\{\omega\in\widetilde{\Omega}:\norm{\mathcal{P}^H_+(\omega) e_i}>0,\quad \forall\ i\in\{1,2,\dots,d\}\bigr\}
			= \\ 
			=\bigcup_{n\in\bNp}\Bigl\{\omega\in\widetilde{\Omega}:\norm{\mathcal{P}^H_+(\omega) e_i}\geq \frac{1}{n},\quad \forall\ i\in\{1,2,\dots,d\}\Bigr\}.
		\end{multline*}
		The Lemma follows from continuity of measure. 
	\end{proof}

	We will then be able to handle $\norm{\mathcal{P}_+^H((i,\alpha)_{0:S-1})e_j}$ using Lemma~\ref{lem: Omega1} and the closeness between $\left(\Phi^H(S,\omega)^{\top} \Phi^H(S,\omega)\right)^{1/2S}$ with $\Lambda(\omega)$ as the former converges to the latter as $S \to \infty$ by Theorem~\ref{thm: MET}. 
	More precisely, denote the singular values of $X\in\bR^{d\times d}$ by $s_1(X)\geq s_2(X)\geq \cdots\geq s_d(X)$. Then for $S\in\bNp$ sufficiently large,
	we have 
	\begin{equation}\label{close: singular value}
		\left|\frac{1}{S}\log s_j\left(\Phi^H(S,\omega)\right)- \lambda_{\mu(j)}\right|=\left|\frac{1}{S}\log s_j\left(\Phi^H\bigl((i,\alpha)_{0:S-1}\bigr)\right)- \lambda_{\mu(j)}\right|\leq \gamma,
	\end{equation}
	for every $j\in\{1,2,\dots,d\}$, where $\lambda_1>\lambda_2>\dots>\lambda_p$ are the Lyapunov exponents from Theorem~\ref{thm: MET}, $\gamma$ is given by \eqref{lambplus-gamma}, and the map $\mu:\{1,2,\dots,d\}\rightarrow\{1,2,\dots,p\}$ satisfies that $\mu(j)=i$ if and only if $d_1+\cdots +d_{i-1}<j\leq d_1+\cdots +d_i$, so $\mu$ corresponds the index for the singular values with that of the Lyapunov exponents. 
	Moreover,
	the convergence also implies that 
	\begin{equation*}
		\norm{\mathcal{P}_+^H(S,\omega)-\mathcal{P}_+^H(\omega)}\leq\frac{\delta}{2},
	\end{equation*}
	for sufficiently large $S$, which then leads to
	\begin{equation}\label{close: singular space}
		\norm{\mathcal{P}^H_+\left(S,\omega\right)e_j}=\norm{\mathcal{P}^H_+\left((i,\alpha)_{0:S-1}\right)e_j}\geq \frac{\delta}{2},
	\end{equation}
	for every $j\in\{1,2,\dots,d\}$,
	where $\mathcal{P}^H_+(S,\omega)$ is the projection operator onto the subspace spanned by the right singular vectors of $\Phi^H(S,\omega)$ corresponding to $d_+$ largest singular values. Let 
	\begin{equation}
		\Omega^S=\bigl\{(i,\alpha)_{0:S-1}\in \Omega_0\times\cdots\times\Omega_{S-1}:\eqref{close: singular value}\text{ and }\eqref{close: singular space}\text{ hold}\bigr\}.
	\end{equation}
	The following lemma states that $\Omega^S$ has high probability for sufficiently large $S$, where with slight abuse of notation, we write $\bP(\Omega^S)=\bP\left(\Omega^S\times\left(\bigtimes_{t\geq S}\Omega_t\right)\right)$.

	\begin{lemma}\label{lem: T-Omega2}
		Under the same assumptions of Lemma~\ref{lem: Omega1}, there exists some $S_*>0$, such that for every $S \in \bNp, S \geq S_*$, it holds $\bP(\Omega^S)\geq 1-2\epsilon$.
	\end{lemma}
	
	\begin{proof}
		For $a.e.\ \omega\in\Omega$, it follows from Theorem~\ref{thm: MET}, in particular \eqref{conv: Lamb}, and standard matrix perturbation analysis (see e.g., \cite[Theorem VI.2.1, Theorem VII.3.1]{bhatia}) that 
		\begin{equation}\label{conv_lya_component}
			\frac{1}{S}s_j(\Phi^H(S,\omega))\rightarrow\lambda_{\mu(j)},\quad S\rightarrow\infty.
		\end{equation}
		for any $j\in\{1,2,\dots,d\}$, and that
		\begin{equation}\label{conv_subspace}
			\mathcal{P}^H_+(S,\omega)\rightarrow\mathcal{P}^H_+(\omega),\quad S\rightarrow\infty.
		\end{equation}
		
		By Egorov's theorem, there exists $\Omega_2^\epsilon\subset\Omega_1^\epsilon$ with $\bP(\Omega_2^\epsilon) \geq 1 - 2 \epsilon$, such that the convergences in \eqref{conv_lya_component} and \eqref{conv_subspace} are both uniform on $\Omega_2^\epsilon$. Here $\Omega_1^\epsilon$ is as in Lemma~\ref{lem: Omega1}. Therefore, for some $S_*$ sufficiently large, we have
		\begin{equation*}\label{uniconv: singval}
			\left|\frac{1}{S}\log s_j\left(\Phi^H(S,\omega)\right)- \lambda_{\mu(j)}\right|\leq \gamma,\quad\forall\ j\in\{1,2,\dots,d\},\ S\geq S_*,\ \omega\in\Omega_2^\epsilon,
		\end{equation*}
		and
		\begin{equation}\label{uniconv: singspace}
			\norm{\mathcal{P}^H_+(S,\omega)-\mathcal{P}^H_+(\omega)}\leq \frac{\delta}{2},\quad\forall\ S\geq S_*,\ \omega\in\Omega_2^\epsilon.
		\end{equation}
		Combining Lemma~\ref{lem: Omega1} and \eqref{uniconv: singspace}, we obtain that
		\begin{equation*}
			\norm{\mathcal{P}^H_+(S,\omega)e_i}\geq \frac{\delta}{2},\quad \forall\ i\in\{1,2,\dots,d\},\ \forall\ S\geq S_*,\ \omega\in\Omega_2^\epsilon.
		\end{equation*}
		
		For any $S\geq S_*$, 
		by the definition of $\Omega^S$, it holds that 
		\begin{equation*}
			\Omega_2^\epsilon\subset\Omega^S\times\Bigl(\bigtimes_{t\geq S}\Omega_t\Bigr),
		\end{equation*}
		which implies the desired estimate 
		\begin{equation*}
			\bP(\Omega^S)=\bP\biggl(\Omega^S\times\Bigl(\bigtimes_{t\geq S}\Omega_t\Bigr)\biggr)\geq\bP(\Omega_2^\epsilon)\geq 1-2\epsilon. %\qedhere
		\end{equation*}
	\end{proof}
	
	Note that $\alpha_{\varrho_t-1}\sim\mathcal{U}_{[\alpha_{\min},\alpha_{\max}]}$ is independent of $\mathcal{F}_{\varrho_t-2}$, $i_{\varrho_t-1}$, and $(i,\alpha)_{\varrho_t:\varrho_t+S-1}$. 
	The next lemma shows that with high probability, the choice of $\alpha_{\varrho_t-1}$ will lead to a nontrivial orthogonal projection of $x_{\varrho_t}$ onto the unstable subspace of $\Phi^H(S,\tau^{\varrho_t}\omega)=\Phi^H\bigl((i,\alpha)_{\varrho_t:\varrho_t+S-1}\bigr)$.
	
	\begin{lemma}\label{lem: nontrial-proj}
		For any $S\in\bNp$, $x_{\varrho_t-1}$, $i_{\varrho_t-1}$, 
		and $(i,\alpha)_{\varrho_t:\varrho_t+S-1} \in\Omega^S$, there exists $I\subset[\alpha_{\min},\alpha_{\max}]$ with $m(I)\geq (1-\epsilon)(\alpha_{\max}-\alpha_{\min})$ where $m(\cdot)$ is the Lebesgue measure,  such that for any $\alpha_{\varrho_t-1}\in I$, it holds that 
		\begin{equation}\label{eq: nontrial-proj}
			\norm{\mathcal{P}^H_+\left((i,\alpha)_{\varrho_t:\varrho_t+S-1}\right)x_{\varrho_t}}\geq \frac{\epsilon\delta (\alpha_{\max}-\alpha_{\min})}{4}\bigl\lvert e_{i_{\varrho_t-1}}^{\top} \nabla f(x_{\varrho_t-1})\bigr\rvert.
		\end{equation}
	\end{lemma}
	
	\begin{proof}
		We assume that $\bigl\lvert e_{i_{\varrho_t-1}}^{\top} \nabla f(x_{\varrho_t-1})\bigr\rvert\neq 0$; otherwise the result is trivial. 
		
		For simplicity of notation, we write 
		\begin{equation*}
			\begin{aligned}
				\mathcal{P}^H_+\left((i,\alpha)_{\varrho_t:\varrho_t+S-1}\right)x_{\varrho_t} & = 
				\mathcal{P}^H_+\left((i,\alpha)_{\varrho_t:\varrho_t+S-1}\right)x_{\varrho_t-1} \\
				& \qquad 
				- \alpha_{\varrho_t - 1} e_{i_{\varrho_t-1}}^{\top}\nabla f(x_{\varrho_t-1})\mathcal{P}^H_+\left((i,\alpha)_{\varrho_t:\varrho_t+S-1}\right)e_{i_{\varrho_t-1}} \\
				& =: y_2-\alpha_{\varrho_t-1} y_1,
			\end{aligned}
		\end{equation*}
		where the last line defines $y_1$ and $y_2$. Using the short-hand notation 
		\begin{equation*}
			r=\frac{\epsilon\delta (\alpha_{\max}-\alpha_{\min})}{4}\bigl\lvert e_{i_{\varrho_t-1}}^{\top} \nabla f(x_{\varrho_t-1})\bigr\rvert, 
		\end{equation*}
		we observe then \eqref{eq: nontrial-proj} holds if and only if $\alpha_{\varrho_t-1} y_1$ is not located in a ball with radius $r$ centered at $y_2$.
		
		It follows from the definition of $\Omega^S$ and \eqref{close: singular space} that $ \norm{\mathcal{P}^H_+\left((i,\alpha)_{\varrho_t:\varrho_t+S-1}\right)e_{i_{\varrho_t-1}}} \geq \frac{\delta}{2}$, 
		which then leads to
		\begin{equation*}
			\norm{y_1}\geq \frac{\delta}{2}\bigl\lvert e_{i_{\varrho_t-1}}^{\top} \nabla f(x_{\varrho_t-1})\bigr\rvert=:\frac{2r}{\epsilon(\alpha_{\max}-\alpha_{\min})}.
		\end{equation*}
		Thus, the set of $\alpha_{\varrho_t-1}$ such that $\alpha_{\varrho_t -1} y_1 \in \mathcal{B}_r(y_2)$ consists of an interval $J$ in $\bR$ with $\|\sup(J)\cdot y_1 - \inf(J)\cdot y_1\|\leq 2r$ as the diameter of $\mathcal{B}_r(y_2)$ is $2r$, which implies $m(J)\leq 2r/\norm{y_1}\leq \epsilon(\alpha_{\max}-\alpha_{\min})$. The lemma is proved then by setting $I=[\alpha_{\max}-\alpha_{\min}]\backslash J$.
	\end{proof}

	With Lemma~\ref{lem: L}--\ref{lem: nontrial-proj}, we now prove Theorem~\ref{thm: local-growth} which relies on approximation of the nonlinear dynamics by linearization and the amplification from the finite block analysis for the linearized system.

	\begin{proof}[Proof of Theorem~\ref{thm: local-growth}]
		Without loss of generality, we will assume $t=0$ in the proof to simplify notation.
		Since $H=\nabla^2 f(x^*)$ is non-degenerate, we can take a neighborhood 
		$U$ of $x^*=0$ and some fixed $\sigma > 0$ such that
		\begin{equation}\label{eq: sigma}
			\norm{\nabla f(x)}\geq \sigma\norm{x},\quad \forall\ x\in U.    
		\end{equation}
		Assumption~\ref{Assump: Hess_bdd} implies that
		\begin{equation*}
			\norm{\nabla f(x)}=\norm{\nabla f(x)-\nabla f(x^*)}\leq M\norm{x-x^*}=M\norm{x},\quad\forall\ x\in\bR^d.
		\end{equation*}
		Using the above inequality and $\alpha_{\max}<1/M$, it holds for every $\omega\in\Omega$ and $t'\in\bN$ that 
		\begin{equation}\label{r-minus}
			\begin{split}
				\norm{x_{t'+1}}&=\norm{x_{t'}-\alpha_{t'} e_{i_{t'}}e_{i_{t'}}^{\top}\nabla f(x_{t'})} \\
				& \geq \norm{x_{t'}}-\alpha_{t'}\norm{e_{i_{t'}}e_{i_{t'}}^{\top}}\cdot \norm{\nabla f(x_{t'})}\\
				&\geq (1-M\alpha_{\max})\norm{x_{t'}},
			\end{split}
		\end{equation}
		and similarly that
		\begin{equation*}
			\norm{x_{t'+1}}\leq(1+M\alpha_{\max})\norm{x_{t'}}.
		\end{equation*}
		We thus define
		\begin{equation*}
			r_-:=1-M\alpha_{\max}\quad\text{and}\quad r_+:=1+M\alpha_{\max}, 
		\end{equation*}
		so that
		\begin{equation}\label{eq:compareiterate}
			r_- \norm{x_{t'}} \leq \norm{x_{t'+1}} \leq r_+ \norm{x_{t'}}.
		\end{equation}
		
		We now choose the time duration $S$ large enough in the finite block analysis to guarantee significant amplification. More specifically, we choose $S$ so large that $S \geq S_*$ ($S_*$ defined in Lemma~\ref{lem: T-Omega2}) and that the following two inequalities hold: 
		\begin{equation}\label{large S1}
			\exp(S(\lambda_+-\gamma))\cdot \frac{\epsilon\delta \mu\sigma (r_-)^{L-1}(\alpha_{\max}-\alpha_{\min})}{8}\geq
			(r_+)^{L},
		\end{equation}
		and
		\begin{equation}\label{large S2}
			(1-6\epsilon)\left(S(\lambda_+-\gamma)+\log \frac{\epsilon\delta \mu\sigma (r_-)^{2(L-1)}(\alpha_{\max}-\alpha_{\min})}{8}\right)+6\epsilon(L+S)\log r_->0,
		\end{equation}
		where $L$ is the upper bound defined in Lemma~\ref{lem: L}, $\mu\leq 1/\sqrt{d}$ is a fixed constant as in Lemma~\ref{lem: L}, $\delta$ is from Lemma~\ref{lem: Omega1}, and $\sigma$ is set in \eqref{eq: sigma}. 
		Thanks to \eqref{def: eps} for our choice of $\epsilon$ and that $\gamma < \lambda_+$ from \eqref{lambplus-gamma}, 
		\eqref{large S1} and \eqref{large S2} are satisfied for sufficiently large $S$. 
		
		Next, we show that, for any $S$ sufficiently large as above, there exists a convex neighborhood $U_1\subset U$ of $x^*=0$, such that for any $t'\in\bN$, any $x_{t'}\in U_1$, and any $(i,\alpha)_{t':t'+S-1}$, it holds that
		\begin{equation}\label{block-error}
			\norm{x_{t'+S}} \geq \norm{\Phi^H\bigl((i,\alpha)_{t':t'+S-1}\bigr)x_{t'}}-\norm{x_{t'}}.
		\end{equation}
		We first define a convex neighborhood $U_0 \subset U$ of $x^*=0$ such that
		\begin{multline*}
			\norm{\left(x-\alpha e_i e_i^{\top}\nabla f(x)\right)-\left(I-\alpha e_i e_i^{\top}H\right)x}=\norm{\alpha e_i e_i^{\top}\left(\nabla f(x)-H x\right)}\\
			=\norm{\alpha e_i e_i^{\top}\int_0^1 (\nabla^2 f(\eta x)-Hx) \mathrm{d} \eta}\leq \frac{1}{S(r_+)^{S-1}} \norm{x},
		\end{multline*}
		for any $x \in U_0$, any $i\in\{1,2,\dots,d\}$, and any $\alpha\in[\alpha_{\min},\alpha_{\max}]$.
		Applying the inequality $S$ times for $x_{t'}\in U_1 = (r_+)^{-(S-1)}U_0$, we have,
		\begin{equation*}
			\begin{split}
				&\norm{x_{t'+S}-\Phi^H\bigl((i,\alpha)_{t':t'+S-1}\bigr)x_{t'}}\\
				\leq & \norm{x_{t'+S}-\left(I-\alpha_{t'+S-1} e_{i_{t'+S-1}}e_{i_{t'+S-1}}^{\top}H\right)x_{t'+S-1}}\\
				&+\norm{I-\alpha_{t'+S-1} e_{i_{t'+S-1}}e_{i_{t'+S-1}}^{\top}H}\cdot \norm{x_{t'+S-1}-\Phi^H\bigl((i,\alpha)_{t':t'+S-2}\bigr)x_{t'}}\\
				\leq & \frac{1}{S(r_+)^{S-1}} \norm{x_{t'+S-1}}+r_+\norm{x_{t'+S-1}-\Phi^H\bigl((i,\alpha)_{t':t'+S-2}\bigr)x_{t'}}\\
				\leq& \frac{1}{S(r_+)^{S-1}} \left(\norm{x_{t'+S-1}}+r_+\norm{x_{t'+S-2}}+\cdots+(r_+)^{S-1}\norm{x_{t'}}\right)\\
				\leq& \norm{x_{t'}},
			\end{split}
		\end{equation*}
		and hence, inequality \eqref{block-error}.
		
		Setting $V=(r_+)^{-(L-1)} U_1$, we then have $x_{\varrho_0-1}\in U$ for any $x_0\in V$, which implies that $\norm{\nabla f(x_{\varrho_0-1})}\geq\sigma\norm{x_{\varrho_0-1}}$ as $\varrho_0\leq L$. 
		According to Lemma~\ref{lem: L}--\ref{lem: nontrial-proj}, for any given $x_0\in V$, with probability at least $1-4\epsilon$, we have $(i,\alpha)_{\varrho_0:\varrho_0+S-1}\in\Omega^S$, and  the followings hold:
		\begin{align}
			\label{ineq1} & \bigl\lvert e_{i_{\varrho_0-1}}^{\top} \nabla f(x_{\varrho_0-1})\bigr\rvert\geq \mu\norm{\nabla f(x_{\varrho_0-1})}, \\
			\label{ineq2} & \norm{\mathcal{P}^H_+\left((i,\alpha)_{\varrho_0:\varrho_0+S-1}\right)x_{\varrho_0}}\geq \frac{\epsilon\delta (\alpha_{\max}-\alpha_{\min})}{4}\bigl\lvert e_{i_{\varrho_0-1}}^{\top} \nabla f(x_{\varrho_0-1})\bigr\rvert,
		\end{align}
		where the probability is the marginal probability on $(i,\alpha)_{0:L+S-1}\in\Omega_0\times\cdots\times \Omega_{L+S-1}$.  
		
		\smallskip 
		
		Recall $\lambda_+$ and $\gamma$ in \eqref{lambplus-gamma}, and $d_+=\sum_{\lambda_i>0}d_i$. 
		It follows from \eqref{close: singular value} of the construction of the set $\Omega^S$ that
		\begin{equation*}
			\frac{1}{S}\log s_j\left(\Phi^H\left((i,\alpha)_{\varrho_0:\varrho_0+S-1}\right)\right)\geq \lambda_{\mu(j)}-\gamma\geq \lambda_+-\gamma,\quad\forall\ j\leq d_+.
		\end{equation*}
		This is to say that the $d_+$ largest singular values of $\Phi^H\left((i,\alpha)_{\varrho_0:\varrho_0+S-1}\right)$
		are all greater than or equal to $\exp(S(\lambda_+-\gamma))$. 
		Therefore, it holds that 
		\begin{equation}\label{linear-growth}
			\begin{split}
				\norm{\Phi^H\left((i,\alpha)_{\varrho_0:\varrho_0+S-1}\right)x_{\varrho_0}}&\geq \norm{\Phi^H\left((i,\alpha)_{\varrho_0:\varrho_0+S-1}\right)\mathcal{P}^H_+\left((i,\alpha)_{\varrho_0:\varrho_0+S-1}\right)x_{\varrho_0}}\\
				&  \leftstackrel{\eqref{ineq2}}{\geq} \exp(S(\lambda_+-\gamma))\cdot \frac{\epsilon\delta (\alpha_{\max}-\alpha_{\min})}{4}\bigl\lvert e_{i_{\varrho_0-1}}^{\top} \nabla f(x_{\varrho_0-1})\bigr\rvert\\
				&  \leftstackrel{\eqref{ineq1}}{\geq}
				\exp(S(\lambda_+-\gamma))\cdot \frac{\epsilon\delta\mu (\alpha_{\max}-\alpha_{\min})}{4}\norm{\nabla f(x_{\varrho_0-1})},
			\end{split}
		\end{equation}
		where the first inequality follows from the fact that $\Phi^H\left((i,\alpha)_{\varrho_0:\varrho_0+S-1}\right)\mathcal{P}^H_+\left((i,\alpha)_{\varrho_0:\varrho_0+S-1}\right)x_{\varrho_0}$ and $\Phi^H\left((i,\alpha)_{\varrho_0:\varrho_0+S-1}\right)\left(I-\mathcal{P}^H_+\left((i,\alpha)_{\varrho_0:\varrho_0+S-1}\right)\right)x_{\varrho_0}$ are orthogonal. Combining \eqref{block-error} and \eqref{linear-growth}, we obtain that
		\begin{equation*}
			\begin{split}
				\norm{x_{\varrho_0+S}} & \geq \exp(S(\lambda_+-\gamma))\cdot \frac{\epsilon\delta \mu(\alpha_{\max}-\alpha_{\min})}{4}\norm{\nabla f(x_{\varrho_0-1})}-\norm{x_{\varrho_0}}\\
				& \leftstackrel{\eqref{eq:compareiterate}}{\geq} \left(\exp(S(\lambda_+-\gamma))\cdot \frac{\epsilon\delta \mu\sigma (r_-)^{L-1}(\alpha_{\max}-\alpha_{\min})}{4}-(r_+)^{L}\right)\cdot\norm{x_0}\\
				& \leftstackrel{\eqref{large S1}}{\geq} \exp(S(\lambda_+-\gamma))\cdot \frac{\epsilon\delta \mu\sigma (r_-)^{L-1}(\alpha_{\max}-\alpha_{\min})}{8}\cdot\norm{x_0}.
			\end{split}
		\end{equation*}
		Therefore, it holds that
		\begin{equation*}
			\begin{aligned}
				\norm{x_{L+S}} & \; \leftstackrel{\eqref{eq:compareiterate}}{\geq} (r_-)^{L-1} \norm{x_{\varrho_0+S}} \\
				& \; \geq \exp(S(\lambda_+-\gamma))\cdot \frac{\epsilon\delta \mu\sigma (r_-)^{2(L-1)}(\alpha_{\max}-\alpha_{\min})}{8}\cdot\norm{x_0}.
			\end{aligned}
		\end{equation*}
		We finally arrive at \eqref{local-growth} by setting $T = L + S$ and combining the above with \eqref{large S2}.
	\end{proof}

	\subsection{Proof of main results}\label{sec: proofmain}

	In this section, we first prove the following theorem, which relies on the local amplification with high probability established in Theorem~\ref{thm: local-growth}. The main result Theorem~\ref{thm: main} will follow as an immediate corollary since $\calA$ is countable and strict saddle points are isolated. 
	
	\begin{theorem}\label{thm: single-pt}
		Suppose that Assumption~\ref{Assump: Hess_bdd}, \ref{Assump: non-degenerate}, and \ref{Assump: non-zero-proj} hold and that $0<\alpha_{\min}<\alpha_{\max}<1/M$. Then for every $x^*\in\calA$ 
		and every $x_0\in\bR^d\backslash\{x^*\}$, it holds that
		\begin{equation*}
			\bP(\calC(x^*,x_0))=0.
		\end{equation*}
	\end{theorem}
	
	\begin{proof}
		Without loss of generality, we assume that $x^*=0$. Conditioned on $\mathcal{F}_{t-1}$ with $x_t\in V$, where $V$ can be assumed to be bounded, Theorem~\ref{thm: local-growth} states that with probability at least $1-4\epsilon$, 
		\begin{equation*}
			\norm{x_{t+T}} \geq A \norm{x_t},
		\end{equation*}
		where to simplify notation we denote 
		\begin{equation*}
			A := \exp\biggl(\frac{6\epsilon}{1 - 6 \epsilon} \bigl\lvert \log(1-M\alpha_{\max}) \bigr\rvert \, T\biggr)
		\end{equation*}
		the amplification factor appearing on the right-hand side of \eqref{local-growth}. 
		Notice also that, due to \eqref{eq:compareiterate}, we always have
		\begin{equation*}
			\norm{x_{t+T}}\geq (r_-)^{T}\norm{x_t}.
		\end{equation*}
		It suffices to show that for any $x_0\in V\backslash\{x^*\}$, with probability $1$ there exists some $t\in\bNp$ such that $x_t\notin V$.
		
		Let us consider the iterates every $T$ steps: Denote $y_t=x_{T t}$ and $\mathcal{G}_t=\mathcal{F}_{T t-1}$ for $t\in\bN$. Denote stopping time
		\begin{equation*}
			\rho=\inf\{t\in\bN:y_t\notin V\},
		\end{equation*}
		it suffices to show that $\bP(\rho<\infty)=1$. We define a sequence of random variables $I_t$ as follows:
		\begin{equation*}
			I_t(\omega)=\begin{cases}1,&\text{if }\norm{y_{t+1}}\geq A\norm{y_t},\\
				0,&\text{otherwise}.\end{cases}
		\end{equation*}
		By the discussion in the beginning of the proof, we have  
		\begin{equation*}
			\bP(I_t = 1\ |\ \mathcal{G}_t, t<\rho)\geq 1-4\epsilon,
		\end{equation*}
		and moreover, setting $S_t(\omega)=\sum_{0\leq s<t} I_t(\omega)$, we have for $t<\rho(\omega)$,
		\begin{equation*}
			\frac{\norm{y_t}}{\norm{y_0}} \geq A^{S_t(\omega)} \cdot (r_-)^{T(t-S_t(\omega))}.
		\end{equation*}
		Denote $R:=\sup_{x\in V}\norm{x}<\infty$. Since  $(1-5\epsilon)\log A+5\epsilon T\log r_->0$, there exists $t_*\in\bN$, such that
		\begin{equation*}
			\left(A^{1-5\epsilon}\cdot (r_-)^{5\epsilon T}\right)^t>\frac{R}{\norm{y_0}},\quad\forall\ t\geq t_*.
		\end{equation*}
		Therefore, for any $t\geq t_*$, it holds that 
		\begin{equation*}
			\bP(\rho> t) = \bP\bigl(\rho> t,S_t\leq (1-5\epsilon)t\bigr)\leq \sum_{i\leq (1-5\epsilon)t} \binom{t}{i} (1-4\epsilon)^i (4\epsilon)^{t-i}.
		\end{equation*}
		As we will show in the next lemma that the right-hand side of above goes to $0$ as $t \to \infty$, and thus $\lim_{t\rightarrow\infty}\bP(\rho> t)=0$, which implies that $\bP(\rho<\infty)=1$.
	\end{proof}
	
	\begin{lemma}\label{lem: binomial}
		For any $\epsilon\in(0,1/4)$, it holds that
		\begin{equation*}
			\lim_{t\rightarrow\infty}\sum_{i\leq(1-5\epsilon)t} \binom{t}{i} (1-4\epsilon)^i (4\epsilon)^{t-i}=0.
		\end{equation*}
	\end{lemma}
	
	\begin{proof}
		Let $X_0,X_1,X_2,\dots$ be a sequence of i.i.d. random variables with $X_i$ being a Bernoulli random variable with expectation $1-4\epsilon$. Denote the average $\bar{X}_t=\frac{1}{t} \sum_{0\leq s<t} X_s$. The weak law of large numbers yields that
		\begin{equation*}
			\sum_{i\leq(1-5\epsilon)t} \binom{t}{i} (1-4\epsilon)^i (4\epsilon)^{t-i}=\bP\left(\bar{X}_t\leq 1-5\epsilon\right)\leq \bP\left(|\bar{X}_t-\bE X_0|\geq \epsilon\right)\rightarrow 0,
		\end{equation*}
		as $t\rightarrow\infty$.
	\end{proof}

	The main theorem then follows directly from Theorem~\ref{thm: single-pt}.
	
	\begin{proof}[Proof of Theorem~\ref{thm: main}]
		Assumption~\ref{Assump: non-degenerate} guarantees that, in a small neighborhood of $x^*$, the gradient $\nabla f(x)=\nabla^2 f(x^*)(x-x^*)+o(\norm{x-x^*})$ is non-vanishing as long as $x\neq x^*$, which implies that $x^*$ is an isolated stationary point. Therefore, $\calA$ is countable. Then Theorem~\ref{thm: main} follows directly from Theorem~\ref{thm: single-pt} and the countability of $\calA$.
	\end{proof}

	\medskip 
	
	We now prove the global convergence, i.e., Corollary~\ref{cor: global-conv}, for which we will show that Algorithm~\ref{alg:random CGD} converges to a critical point of $f$ with some appropriate assumptions. We first show that the limit of each convergent subsequence of $\{x_t\}_{t\in\bN}$ is a critical point of $f$. 
	
	\begin{proposition}\label{prop: converge-stationary}
		If Assumption~\ref{Assump: Hess_bdd} holds and $0<\alpha_{\min}<\alpha_{\max}<1/M$, %\jl{shall we use $1/M$ too?} 
		for any $x_0\in\bR^d$ with bounded level set $L(x_0)=\{x\in\bR^d:f(x)\leq f(x_0)\}$, with probability $1$, every accumulation point of $\{x_t\}_{t\in\bN}$ is in $\crit$.
	\end{proposition}
	
	\begin{proof}
		Algorithm~\ref{alg:random CGD} is always monotone since the following holds for any $t\in\bN$ by Taylor's expansion:
		\begin{equation}\label{f: decay}
			\begin{split}
				f(x_{t+1})& =f\left(x_t-\alpha_t e_{i_t} e_{i_t}^{\top} \nabla f(x_t)\right)\\
				&=f(x_t)-\alpha_t \left(e_{i_t}^{\top} \nabla f(x_t)\right)^2 \\
				&\qquad\qquad+\frac{1}{2}\alpha_t^2\left(e_{i_t}^{\top} \nabla f(x_t)\right)^2 \cdot e_{i_t}^{\top} \nabla f\left(x_t-\theta_t\alpha_t e_{i_t} e_{i_t}^{\top} \nabla f(x_t)\right)e_{i_t}\\
				& \leq f(x_t)-\frac{1}{2}\alpha_t(e_i^{\top} \nabla f(x_t))^2\\
				& \leq  f(x_t),
			\end{split}
		\end{equation}
		where $\theta_t\in(0,1)$, which implies that the whole sequence $\{x_t\}_{t\in\bN}$ is contained in the bounded level set $L(x_0)$.

		Let us consider any $\eta>0$ and set
		\begin{equation*}
			L(x_0,\eta)=\{x\in L(x_0):\norm{\nabla f(x)}\geq \eta\},
		\end{equation*}
		which is either empty or compact. We claim that with probability $1$, the accumulation points of $\{x_t\}_{t\in\bN}$ will not be located in  $L(x_0,\eta)$. This is clear when $L(x_0,\eta)$ is empty, so it suffices to consider compact $L(x_0,\eta)$. Set $\mu \in (0, 1/\sqrt{d}]$ as a fixed constant.  For any $x\in L(x_0,\eta)$, there exists an open neighborhood $U_x$ of $x$ and a coordinate $i_x\in\{1,2,\dots,d\}$, such that
		\begin{equation}\label{def: Ux1}
			\bigl\lvert e_{i_x}^{\top} \nabla f(y)\bigr\rvert\geq \mu \norm{\nabla f(x)} \geq \mu \eta,\quad \forall\ y\in U_x,
		\end{equation}
		and that
		\begin{equation}\label{def: Ux2}
			\sup_{y\in U_x} f(y)-\inf_{y\in U_x} f(y)<\frac{\alpha_{\min}\mu^2\eta^2}{2}.
		\end{equation}
		Noticing that $L(x_0,\eta)\subset\bigcup_{x\in L(x_0,\eta)} U_x$, by the compactness, there exist finitely many points, say $x^1,x^2,\dots,x^K$, such that 
		\begin{equation*}
			L(x_0,\eta)\subset\bigcup_{1\leq k\leq K} U_{x^k}.
		\end{equation*}
		For any $k\in\{1,2,\dots,K\}$, combining \eqref{f: decay}, \eqref{def: Ux1}, and \eqref{def: Ux2}, we know that for any $t$, conditioned on $\mathcal{F}_{t-1}$ with $x_t\in U_{x^k}$, if $i_t=i_x$ (which has probability $1/d$), then $f(x_{t+1})<\inf_{y\in U_{x^k}} f(y)$ and thus $x_{t'}\not\in U_{x^k}$ for all $t' > t$. 
		
		Therefore, the probability that there are infinitely many $t\in\bN$ with $x_t\in U_{x^k}$ is zero, which implies that $\{x_t\}_{t\in\bN}$ does not have accumulation points in $U_{x^k}$ with probability $1$. We conclude that with probability $1$, $L(x_0,\eta)$ does not contain any accumulation points of $\{x_t\}_{t\in\bN}$ as $K$ is finite.  Since this holds for any $\eta> 0$, we have  for $\bP$-a.e.{} $\omega\in\Omega$, $\{x_t\}_{t\in\bN}$ has no accumulation points in any $\bigcup_{n\geq 1}L(x_0,1/n)$, which then leads to the desired result.
	\end{proof}
	
	Proposition~\ref{prop: converge-stationary} implies that any accumulation point of the algorithm iterate is a critical point. If we further assume that each critical point of $f$ is isolated, we would conclude that the whole sequence $\{x_t\}_{t\in\bN}$ converges and the limit is in $\crit$.
	
	\begin{proposition}\label{prop: converge-stationary2}
		Under the  assumptions of  Proposition~\ref{prop: converge-stationary}. If every $x^*\in\crit$ is an isolated critical point of $f$, then with probability $1$, $x_t$ converges to some critical point of $f$ as $t\rightarrow\infty$.
	\end{proposition}
	
	\begin{proof}
		It follows from Proposition~\ref{prop: converge-stationary} that $\norm{\nabla f(x_t)}$ converges to $0$ as $t\rightarrow\infty$ for $a.e.$ $\omega\in\Omega$. In fact, if there were a subsequence $\{x_{t_k}\}_{k\in\bN}$ and $\epsilon>0$ with $\norm{\nabla f(x_{t_k})}\geq \epsilon,\ \forall\ k\in\bN$. Then by the boundedness of $L(x_0)$, $\{x_{t_k}\}_{k\in\bN}$ would have some accumulation point which is not a stationary point of $f$, which leads to a contradiction.

		Moreover, $\crit\cap L(x_0)$ is a finite set, since otherwise, $\crit\cap L (x_0)$ would have a limiting point which would be a non-isolated critical point of $f$, violating the assumption. 
		
		Consider a fixed $\omega\in\Omega$ with $\lim_{t\rightarrow\infty}\norm{\nabla f(x_t)}=0$. Select a open neighborhood $U_{x^*}$ for every $x^*\in\crit\cap L(x_0)$, such that there exists some $\delta>0$ with
		\begin{equation*}
			\text{dist}(U_{x^*},U_{y^*})=\inf_{x\in U_{x^*},y\in U_{y^*}}\norm{x-y}>\delta,\quad \forall\ x^*,y^*\in\crit.
		\end{equation*}
		If $\{x_t\}_{t\in\bN}$ has more than one accumulation point,  there would be infinitely many iterates located in $L(x_0)\backslash\bigcup_{x^*\in\crit\cap L(x_0)} U_{x^*}$ which is compact. Therefore, $\{x_t\}_{t\in\bN}$ would have an accumulation point in $L(x_0)\backslash\bigcup_{x^*\in\crit\cap L(x_0)} U_{x^*}$, which contradicts Proposition~\ref{prop: converge-stationary}.
	\end{proof}
	
	Corollary~\ref{cor: global-conv} is now an immediate consequence.
	
	\begin{proof}[Proof of Corollary~\ref{cor: global-conv}]
		The result follows directly from Theorem~\ref{thm: main} and Proposition~\ref{prop: converge-stationary2}.
	\end{proof}

	\bibliographystyle{amsxport}
	\bibliography{references}

\begin{bibdiv}
\begin{biblist}

\bib{Arnold_RDS}{book}{
      author={Arnold, Ludwig},
       title={Random dynamical systems},
      series={Springer monographs in mathematics},
   publisher={Springer},
     address={New York},
        date={1998},
        ISBN={3540637583},
}

\bib{attouch2010proximal}{article}{
      author={Attouch, H{\'e}dy},
      author={Bolte, J{\'e}r{\^o}me},
      author={Redont, Patrick},
      author={Soubeyran, Antoine},
       title={Proximal alternating minimization and projection methods for
  nonconvex problems: An approach based on the kurdyka-{\l}ojasiewicz
  inequality},
        date={2010},
     journal={Mathematics of operations research},
      volume={35},
      number={2},
       pages={438\ndash 457},
}

\bib{attouch2013convergence}{article}{
      author={Attouch, Hedy},
      author={Bolte, J{\'e}r{\^o}me},
      author={Svaiter, Benar~Fux},
       title={Convergence of descent methods for semi-algebraic and tame
  problems: proximal algorithms, forward--backward splitting, and regularized
  gauss--seidel methods},
        date={2013},
     journal={Mathematical Programming},
      volume={137},
      number={1},
       pages={91\ndash 129},
}

\bib{Beck-13}{article}{
      author={Beck, Amir},
      author={Tetruashvili, Luba},
       title={On the convergence of block coordinate descent type methods},
        date={2013},
     journal={SIAM journal on Optimization},
      volume={23},
      number={4},
       pages={2037\ndash 2060},
}

\bib{bhatia}{book}{
      author={Bhatia, Rajendra},
       title={Matrix analysis},
   publisher={Springer Science \& Business Media},
        date={2013},
      volume={169},
}

\bib{bolte2014proximal}{article}{
      author={Bolte, J{\'e}r{\^o}me},
      author={Sabach, Shoham},
      author={Teboulle, Marc},
       title={Proximal alternating linearized minimization for nonconvex and
  nonsmooth problems},
        date={2014},
     journal={Mathematical Programming},
      volume={146},
      number={1},
       pages={459\ndash 494},
}

\bib{boct2020inertial}{article}{
      author={Bo{\c{t}}, Radu~Ioan},
      author={Dao, Minh~N},
      author={Li, Guoyin},
       title={Inertial proximal block coordinate method for a class of
  nonsmooth and nonconvex sum-of-ratios optimization problems},
        date={2020},
     journal={arXiv preprint arXiv:2011.09782},
}

\bib{Boxler-89-center}{article}{
      author={Boxler, Petra},
       title={A stochastic version of center manifold theory},
        date={1989},
     journal={Probability Theory and Related Fields},
      volume={83},
      number={4},
       pages={509\ndash 545},
}

\bib{Du-17}{inproceedings}{
      author={Du, Simon~S},
      author={Jin, Chi},
      author={Lee, Jason~D},
      author={Jordan, Michael~I},
      author={Singh, Aarti},
      author={Poczos, Barnabas},
       title={Gradient descent can take exponential time to escape saddle
  points},
        date={2017},
   booktitle={Advances in neural information processing systems},
       pages={1067\ndash 1077},
}

\bib{Froyland-15}{article}{
      author={Froyland, Gary},
      author={Gonz{\'a}lez-Tokman, Cecilia},
      author={Quas, Anthony},
       title={Stochastic stability of lyapunov exponents and oseledets
  splittings for semi-invertible matrix cocycles},
        date={2015},
     journal={Communications on Pure and Applied Mathematics},
      volume={68},
      number={11},
       pages={2052\ndash 2081},
}

\bib{Ge-15}{inproceedings}{
      author={Ge, Rong},
      author={Huang, Furong},
      author={Jin, Chi},
      author={Yuan, Yang},
       title={Escaping from saddle points—online stochastic gradient for
  tensor decomposition},
        date={2015},
   booktitle={{Conference on Learning Theory}},
       pages={797\ndash 842},
}

\bib{ge2017no}{inproceedings}{
      author={Ge, Rong},
      author={Jin, Chi},
      author={Zheng, Yi},
       title={No spurious local minima in nonconvex low rank problems: A
  unified geometric analysis},
organization={PMLR},
        date={2017},
   booktitle={International conference on machine learning},
       pages={1233\ndash 1242},
}

\bib{Guo-16}{article}{
      author={Guo, Peng},
      author={Shen, Jun},
       title={Smooth center manifolds for random dynamical systems},
        date={2016},
     journal={Rocky Mountain Journal of Mathematics},
      volume={46},
      number={6},
       pages={1925\ndash 1962},
}

\bib{Guo-20}{article}{
      author={Guo, Xin},
      author={Han, Jiequn},
      author={Tang, Wenpin},
       title={Perturbed gradient descent with occupation time},
        date={2020},
     journal={arXiv preprint arXiv:2005.04507},
}

\bib{Gurbuzbalaban-20}{article}{
      author={G{\"u}rb{\"u}zbalaban, Mert},
      author={Ozdaglar, Asuman},
      author={Vanli, Nuri~Denizcan},
      author={Wright, Stephen~J},
       title={Randomness and permutations in coordinate descent methods},
        date={2020},
     journal={Mathematical Programming},
      volume={181},
      number={2},
       pages={349\ndash 376},
}

\bib{Jin-17}{inproceedings}{
      author={Jin, Chi},
      author={Ge, Rong},
      author={Netrapalli, Praneeth},
      author={Kakade, Sham~M},
      author={Jordan, Michael~I},
       title={How to escape saddle points efficiently},
        date={2017},
   booktitle={{International Conference on Machine Learning}},
       pages={1724\ndash 1732},
}

\bib{Jin-19}{article}{
      author={Jin, Chi},
      author={Netrapalli, Praneeth},
      author={Ge, Rong},
      author={Kakade, Sham~M},
      author={Jordan, Michael~I},
       title={On nonconvex optimization for machine learning: Gradients,
  stochasticity, and saddle points},
        date={2021},
     journal={Journal of the ACM (JACM)},
      volume={68},
      number={2},
       pages={1\ndash 29},
}

\bib{Jin-18}{inproceedings}{
      author={Jin, Chi},
      author={Netrapalli, Praneeth},
      author={Jordan, Michael~I},
       title={Accelerated gradient descent escapes saddle points faster than
  gradient descent},
organization={PMLR},
        date={2018},
   booktitle={{Conference On Learning Theory}},
       pages={1042\ndash 1085},
}

\bib{kawaguchi2016deep}{article}{
      author={Kawaguchi, Kenji},
       title={Deep learning without poor local minima},
        date={2016},
     journal={Advances in neural information processing systems},
      volume={29},
}

\bib{Ledrappier-91}{article}{
      author={Ledrappier, F.},
      author={Young, Lai-Sang},
       title={Stability of lyapunov exponents},
        date={1991},
     journal={Ergodic Theory and Dynamical Systems},
      volume={11},
      number={3},
       pages={469\ndash 484},
}

\bib{Lee-19}{article}{
      author={Lee, Ching-Pei},
      author={Wright, Stephen~J},
       title={Random permutations fix a worst case for cyclic coordinate
  descent},
        date={2019},
     journal={IMA Journal of Numerical Analysis},
      volume={39},
      number={3},
       pages={1246\ndash 1275},
}

\bib{Lee-2019}{article}{
      author={Lee, Jason~D},
      author={Panageas, Ioannis},
      author={Piliouras, Georgios},
      author={Simchowitz, Max},
      author={Jordan, Michael~I},
      author={Recht, Benjamin},
       title={First-order methods almost always avoid strict saddle points},
        date={2019},
        ISSN={1436-4646},
     journal={Mathematical programming},
      volume={176},
      number={1-2},
       pages={311\ndash 337},
}

\bib{Lee-16}{inproceedings}{
      author={Lee, Jason~D},
      author={Simchowitz, Max},
      author={Jordan, Michael~I},
      author={Recht, Benjamin},
       title={Gradient descent only converges to minimizers},
        date={2016},
   booktitle={{Conference on Learning Theory}},
       pages={1246\ndash 1257},
}

\bib{Levy-16}{article}{
      author={Levy, Kfir~Y},
       title={The power of normalization: Faster evasion of saddle points},
        date={2016},
     journal={arXiv preprint arXiv:1611.04831},
}

\bib{Li-13}{article}{
      author={Li, Ji},
      author={Lu, Kening},
      author={Bates, Peter},
       title={Normally hyperbolic invariant manifolds for random dynamical
  systems: Part i-persistence},
        date={2013},
     journal={Transactions of the American Mathematical Society},
       pages={5933\ndash 5966},
}

\bib{Weigu-05}{article}{
      author={Li, Weigu},
      author={Lu, Kening},
       title={Poincar{\'e} theorems for random dynamical systems},
        date={2005},
     journal={Ergodic Theory and Dynamical Systems},
      volume={25},
      number={4},
       pages={1221},
}

\bib{Weigu-Sternberg}{article}{
      author={Li, Weigu},
      author={Lu, Kening},
       title={Sternberg theorems for random dynamical systems},
        date={2005},
        ISSN={0010-3640},
     journal={Communications on Pure and Applied Mathematics},
      volume={58},
      number={7},
       pages={941\ndash 988},
}

\bib{Weigu-08}{article}{
      author={Li, Weigu},
      author={Lu, Kening},
       title={A siegel theorem for dynamical systems under random
  perturbations},
        date={2008},
     journal={Discrete \& Continuous Dynamical Systems-B},
      volume={9},
      number={3\&4, May},
       pages={635},
}

\bib{Weigu-16}{article}{
      author={Li, Weigu},
      author={Lu, Kening},
       title={Takens theorem for random dynamical systems},
        date={2016},
     journal={Discrete \& Continuous Dynamical Systems-B},
      volume={21},
      number={9},
       pages={3191},
}

\bib{leadeigen-19}{article}{
      author={Li, Yingzhou},
      author={Lu, Jianfeng},
      author={Wang, Zhe},
       title={Coordinatewise descent methods for leading eigenvalue problem},
        date={2019},
     journal={SIAM Journal on Scientific Computing},
      volume={41},
      number={4},
       pages={A2681\ndash A2716},
}

\bib{Lian-10}{book}{
      author={Lian, Zeng},
      author={Lu, Kening},
       title={Lyapunov exponents and invariant manifolds for random dynamical
  systems in a banach space},
   publisher={American Mathematical Soc.},
        date={2010},
}

\bib{Liu-15}{article}{
      author={Liu, Ji},
      author={Wright, Stephen~J},
       title={Asynchronous stochastic coordinate descent: Parallelism and
  convergence properties},
        date={2015},
     journal={SIAM Journal on Optimization},
      volume={25},
      number={1},
       pages={351\ndash 376},
}

\bib{Liu-14}{inproceedings}{
      author={Liu, Ji},
      author={Wright, Stephen~J},
      author={R{\'e}, Christopher},
      author={Bittorf, Victor},
      author={Sridhar, Srikrishna},
       title={An asynchronous parallel stochastic coordinate descent
  algorithm},
        date={2014},
   booktitle={{International Conference on Machine Learning}},
       pages={469\ndash 477},
}

\bib{Liu-06}{book}{
      author={Liu, Pei-Dong},
      author={Qian, Min},
       title={Smooth ergodic theory of random dynamical systems},
   publisher={Springer},
        date={2006},
}

\bib{lu2017depth}{article}{
      author={Lu, Haihao},
      author={Kawaguchi, Kenji},
       title={Depth creates no bad local minima},
        date={2017},
     journal={arXiv preprint arXiv:1702.08580},
}

\bib{Nesterov-12}{article}{
      author={Nesterov, Yu},
       title={Efficiency of coordinate descent methods on huge-scale
  optimization problems},
        date={2012},
     journal={SIAM Journal on Optimization},
      volume={22},
      number={2},
       pages={341\ndash 362},
}

\bib{oseledets}{article}{
      author={Oseledets, Valery~Iustinovich},
       title={A multiplicative ergodic theorem. liapunov characteristic numbers
  for dynamical systems},
        date={1968},
     journal={Trans. Moscow Math. Soc.},
      volume={19},
       pages={197\ndash 231},
}

\bib{Oswald-17}{article}{
      author={Oswald, Peter},
      author={Zhou, Weiqi},
       title={Random reordering in sor-type methods},
        date={2017},
     journal={Numerische Mathematik},
      volume={135},
      number={4},
       pages={1207\ndash 1220},
}

\bib{ONeill-19}{article}{
      author={O’Neill, Michael},
      author={Wright, Stephen~J},
       title={Behavior of accelerated gradient methods near critical points of
  nonconvex functions},
        date={2019},
     journal={Mathematical Programming},
      volume={176},
      number={1-2},
       pages={403\ndash 427},
}

\bib{raghunathan}{article}{
      author={Raghunathan, Madabusi~S},
       title={A proof of oseledec’s multiplicative ergodic theorem},
        date={1979},
     journal={Israel Journal of Mathematics},
      volume={32},
      number={4},
       pages={356\ndash 362},
}

\bib{Richtarik-14}{article}{
      author={Richt{\'a}rik, Peter},
      author={Tak{\'a}{\v{c}}, Martin},
       title={Iteration complexity of randomized block-coordinate descent
  methods for minimizing a composite function},
        date={2014},
     journal={Mathematical Programming},
      volume={144},
      number={1-2},
       pages={1\ndash 38},
}

\bib{Richtarik-16}{article}{
      author={Richt{\'a}rik, Peter},
      author={Tak{\'a}{\v{c}}, Martin},
       title={Parallel coordinate descent methods for big data optimization},
        date={2016},
     journal={Mathematical Programming},
      volume={156},
      number={1-2},
       pages={433\ndash 484},
}

\bib{Ruelle-1979}{article}{
      author={Ruelle, David},
       title={Ergodic theory of differentiable dynamical systems},
        date={1979},
        ISSN={1618-1913},
     journal={Publications mathématiques. Institut des hautes études
  scientifiques},
      volume={50},
      number={1},
       pages={27\ndash 58},
}

\bib{Ruelle-82}{article}{
      author={Ruelle, David},
       title={Characteristic exponents and invariant manifolds in hilbert
  space},
        date={1982},
     journal={Annals of Mathematics},
       pages={243\ndash 290},
}

\bib{Saha-13}{article}{
      author={Saha, Ankan},
      author={Tewari, Ambuj},
       title={On the nonasymptotic convergence of cyclic coordinate descent
  methods},
        date={2013},
     journal={SIAM Journal on Optimization},
      volume={23},
      number={1},
       pages={576\ndash 601},
}

\bib{shi2016primer}{article}{
      author={Shi, Hao-Jun~Michael},
      author={Tu, Shenyinying},
      author={Xu, Yangyang},
      author={Yin, Wotao},
       title={A primer on coordinate descent algorithms},
        date={2016},
     journal={arXiv preprint arXiv:1610.00040},
}

\bib{Shub-87}{book}{
      author={Shub, Michael},
       title={Global stability of dynamical systems},
   publisher={Springer Science \& Business Media},
        date={1987},
}

\bib{sun2018geometric}{article}{
      author={Sun, Ju},
      author={Qu, Qing},
      author={Wright, John},
       title={A geometric analysis of phase retrieval},
        date={2018},
     journal={Foundations of Computational Mathematics},
      volume={18},
      number={5},
       pages={1131\ndash 1198},
}

\bib{Sun-15}{article}{
      author={Sun, Ruoyu},
      author={Hong, Mingyi},
       title={Improved iteration complexity bounds of cyclic block coordinate
  descent for convex problems},
        date={2015},
     journal={Advances in Neural Information Processing Systems},
      volume={28},
       pages={1306\ndash 1314},
}

\bib{Sun-19}{article}{
      author={Sun, Ruoyu},
      author={Ye, Yinyu},
       title={Worst-case complexity of cyclic coordinate descent: $o(n^2)$ gap
  with randomized version},
        date={2019},
     journal={Mathematical Programming},
       pages={1\ndash 34},
}

\bib{Tseng-09}{article}{
      author={Tseng, Paul},
      author={Yun, Sangwoon},
       title={A coordinate gradient descent method for nonsmooth separable
  minimization},
        date={2009},
     journal={Mathematical Programming},
      volume={117},
      number={1-2},
       pages={387\ndash 423},
}

\bib{Walters-93}{article}{
      author={Walters, Peter},
       title={A dynamical proof of the multiplicative ergodic theorem},
        date={1993},
     journal={Transactions of the American Mathematical Society},
      volume={335},
      number={1},
       pages={245\ndash 257},
}

\bib{Wanner-95}{incollection}{
      author={Wanner, Thomas},
       title={Linearization of random dynamical systems},
        date={1995},
   booktitle={Dynamics reported},
   publisher={Springer},
       pages={203\ndash 268},
}

\bib{Wright-12}{article}{
      author={Wright, Stephen~J},
       title={Accelerated block-coordinate relaxation for regularized
  optimization},
        date={2012},
     journal={SIAM Journal on Optimization},
      volume={22},
      number={1},
       pages={159\ndash 186},
}

\bib{Wright-15}{article}{
      author={Wright, Stephen~J},
       title={Coordinate descent algorithms},
        date={2015},
     journal={Mathematical Programming},
      volume={151},
      number={1},
       pages={3\ndash 34},
}

\bib{Wright-20}{article}{
      author={Wright, Stephen~J},
      author={Lee, Ching-pei},
       title={Analyzing random permutations for cyclic coordinate descent},
        date={2020},
     journal={Mathematics of Computation},
}

\end{biblist}
\end{bibdiv}
	
	\appendix
	
	\section{Validity of Assumption~\ref{Assump: non-zero-proj}}\label{sec: ex-asp3}
	
	In this appendix, we provide some justification of Assumption~\ref{Assump: non-zero-proj}, which is expected to hold generically. In particular, the following proposition validates this assumption when the off-diagonal entries of $H$ are all non-zero. 
	
	\begin{proposition}\label{prop-ex-asp3} Suppose that the largest Lyapunov exponent of $\Phi^H(t,\omega)$ is positive. Then Assumption~\ref{Assump: non-zero-proj} holds as long as $1<\alpha_{\min}<\alpha_{\max}<1/\max_{1\leq i\leq d}|H_{ii}|$ and every off-diagonal entry of $H$ is non-zero.
	\end{proposition}
	
	\begin{proof}
		For any element $\omega$ in $\Omega$, we take the smallest $\ell$ such that   $\{1,2,\dots,d\}=\{i_0,i_1,\dots,i_{\ell-1}\}$ and write 
		\begin{equation*}
			\omega=((i_0,\alpha_0),\dots,(i_{\ell-1},\alpha_{\ell-1}),\omega'),
		\end{equation*}
		where $\omega' = \tau^\ell\omega \in \Omega$. We have that $\ell$ is finite for a.e. $\omega\in\Omega$. Note that we can view $\ell-1$ as a stopping time, in particular, given $\ell$, $\omega'$ has distribution $\bP$ and is independent with $\mathcal{F}_{\ell-1}$.
		
		Let $\{v_1',v_2',\dots,v_m'\}$ be a set of basis vectors for $W^H_-(\omega')=W^H_-(\tau^\ell\omega)$. Then a set of basis vectors for $W^H_-(\omega)$ is given by
		\begin{equation*}\label{basis-vj}
			v_j=\bigl(I-\alpha_0 e_{i_0}e_{i_0}^{\top} H\bigr)^{-1}\cdots\bigl(I-\alpha_{\ell-1} e_{i_{\ell-1}}e_{i_{\ell-1}}^{\top} H\bigr)^{-1}v_j',\quad j=1,2,\dots,m.
		\end{equation*}
		Denote the matrices concatenated by the column vectors as $V' = \bigl( v_1'|v_2'|\cdots|v_m'\bigr)$ and $V=\bigl( v_1|v_2|\cdots|v_m\bigr)$. If $e_i\in W^H_-(\omega)=\text{span}\{v_1,v_2,\dots,v_m\}$, then  $V_{\hat{\imath},:}$ is column-rank deficient since the existence of a positive Lyapunov exponent implies that $m\leq d-1$. Here and for the rest of the appendix, we denote by $V_{\hat{\imath},:}$ the $(d-1)\times m$ matrix obtained via removing $i$-th row of $V\in\bR^{d\times m}$.
		
		Therefore, as Assumption~\ref{Assump: non-zero-proj} is equivalent to that $e_i\notin W^H_-(\omega)$ holds for any $i\in\{1,2,\dots,d\}$ and almost every $\omega\in\Omega$, it suffices to show that $V_{\hat{\imath},:}$ has full column-rank with probability $1$. The key point is that given $\ell$,  $\alpha_0,\alpha_1,\dots,\alpha_{\ell-1}$ are independent with  $i_0,i_1,\dots,i_{\ell-1}$ and $\omega'=\tau^\ell\omega$. Thus, it suffices to show that with fixed $\ell$, $i_0,i_1,\dots,i_{\ell-1}$, $\omega'=\tau^\ell\omega$, and $v_1',v_2',\dots,v_m'$, the set of all $\alpha_0,\alpha_1,\dots,\alpha_{\ell-1}$ that yield the rank-deficiency of $V_{\hat{\imath},:}$ is of measure zero; and without loss of generality, we can assume $i = 1$. 
		Noticing that  $i_0,i_1,\cdots,i_{\ell-1}$ cover all the coordinates and that every off-diagonal entry of $H$ is non-zero, the desired result follows directly from the following Lemma~\ref{lem: full-rank} applied repeatedly.
	\end{proof}
	
	\begin{lemma}\label{lem: full-rank}
		Suppose that $X=\bigl( X_1|X_2|\cdots|X_d\bigr)^{\top}$ and $Y=\bigl(Y_1|Y_2|\cdots|Y_d\bigr)^{\top}$ are full-column-rank $d\times m$ matrices satisfying $Y=(I-\alpha e_k e_k^{\top} H)^{-1}X$ (we suppress in the notation the dependence of $Y$ on $k$ and $\alpha$ for simplicity). 
		Then the followings holds:
		\begin{itemize}
			\item[(i)] If $X_{\hat{1},:}$ has full column-rank, then for any $k=\{1,2,\dots,d\}$, $Y_{\hat{1},:}$  also has full column-rank for a.e. $\alpha$.
			\item[(ii)] Suppose that $X_{\hat{1},:}$ is column-rank deficient, and let $2\leq j_1<j_2<\dots<j_{m-1}\leq d$ be row indices such that 
			\begin{equation*}
				X_j\in\emph{span}\{X_{j_1},X_{j_2},\dots,X_{j_{m-1}}\},\quad\forall\ j\in\{2,3,\dots,d\}.
			\end{equation*}
			If $k\in\{1,j_1,j_2,\dots,j_{m-1}\}$, then we have with probability $1$, either $Y_{\hat{1},:}$ has full column-rank or $Y_{\hat{1},:}$ is column-rank deficient with
			\begin{equation*}
				Y_j\in\emph{span}\{Y_{j_1},Y_{j_2},\dots,Y_{j_{m-1}}\},\quad\forall\ j\in\{2,3,\dots,d\}.
			\end{equation*}
			If $k\notin\{1,j_1,j_2,\dots,j_{m-1}\}$ and $H_{k1}\neq 0$, then $Y_{\hat{1},:}$ has full column-rank.
		\end{itemize}
	\end{lemma}
	
	\begin{proof}[Proof of Lemma~\ref{lem: full-rank}]
		By \eqref{inverse-AH}, it holds that $Y_j=X_j$ for $j\neq k$ and that
		\begin{equation*}
			Y_k=\frac{1}{1-\alpha H_{kk}}\left(X_k+\alpha\sum_{j\neq k}H_{k j}X_j\right).
		\end{equation*}
		
		For point 
		(i), we notice that if $k=1$, then $Y_{\hat{1},:}=X_{\hat{1},:}$ has full column-rank. If $k>1$, then it follows from $X_1\in\text{span}\{X_2,\dots,X_d\}$ that $Y_{\hat{1},:}$ also has full column-rank for a.e.~$\alpha$.
		
		For point (ii). We have
		\begin{equation*}
			\text{span}\{X_1,X_{j_1},X_{j_2},\dots,X_{j_{m-1}}\}=\bR^m.
		\end{equation*}
		If $k\in\{1,j_1,j_2,\dots,j_{m-1}\}$, then $\text{span}\{Y_1,Y_{j_1},Y_{j_2},\dots,Y_{j_{m-1}}\}=\bR^m$ holds for $a.e.$ $\alpha$. Therefore, we obtain that $Y_{j_1},Y_{j_2},\dots,Y_{j_{m-1}}$ are linearly independent, which implies that either $Y_{\hat{1},:}$ has full column-rank or 
		\begin{equation*}
			Y_j\in\text{span}\{Y_{j_1},Y_{j_2},\dots,Y_{j_{m-1}}\},\quad\forall\ j\in\{2,3,\dots,n\}.
		\end{equation*}
		If $k\notin\{j_1,j_2,\dots,j_{m-1}\}$, then $Y_{\hat{1},:}$ has full column-rank since $H_{k1}\neq 0$.
	\end{proof}
	
\end{document}